\documentclass[11pt,oneside,english,british,reqno]{amsart}
\usepackage[latin9]{inputenc}
\usepackage{xcolor}
\usepackage{babel}
\usepackage{amsthm}
\usepackage{amsbsy}
\usepackage{amstext}
\usepackage{amssymb}
\usepackage{esint}
\usepackage[all]{xy}
\PassOptionsToPackage{normalem}{ulem}
\usepackage{ulem}
\usepackage[unicode=true,pdfusetitle,
 bookmarks=true,bookmarksnumbered=false,bookmarksopen=false,
 breaklinks=false,pdfborder={0 0 1},backref=false,colorlinks=true]
 {hyperref}
\usepackage{breakurl}

\makeatletter

\providecolor{lyxadded}{rgb}{0,0,1}
\providecolor{lyxdeleted}{rgb}{1,0,0}

\numberwithin{equation}{section}
\numberwithin{figure}{section}
\usepackage{enumitem}

\newlength{\lyxlistindent}      
\theoremstyle{plain}
\newtheorem{thm}{\protect\theoremname}
  \theoremstyle{definition}
  \newtheorem{defn}[thm]{\protect\definitionname}
  \theoremstyle{remark}
  \newtheorem{rem}[thm]{\protect\remarkname}
  \theoremstyle{definition}
  \newtheorem{example}[thm]{\protect\examplename}
  \theoremstyle{plain}
  \newtheorem{cor}[thm]{\protect\corollaryname}
  \theoremstyle{plain}
  \newtheorem{lem}[thm]{\protect\lemmaname}
  \theoremstyle{plain}
  \newtheorem*{fact*}{\protect\factname}

\@ifundefined{date}{}{\date{}}

\usepackage{babel}
\usepackage{amscd}
\usepackage[all]{xy}

\newcommand{\N}{\ensuremath{\mathbb{N}}}
\newcommand{\Z}{\ensuremath{\mathbb{Z}}}
\newcommand{\Q}{\ensuremath{\mathbb{Q}}}
\newcommand{\R}{\ensuremath{\mathbb{R}}}
\newcommand{\Co}{\ensuremath{\mathbb{C}}}
\DeclareMathOperator*{\colim}{colim}
\newcommand{\ra}{\longrightarrow}
\newcommand{\F}{\mathcal{F}}
\newcommand{\D}{\mathcal{D}}
\newcommand{\M}{\mathcal{M}}
\newcommand{\mC}{\mathcal{C}}
\newcommand{\supp}{\text{supp}}

\newcommand{\Coo}{\mbox{\ensuremath{\mathcal{C}}}^{\infty}}
\newcommand{\Diff}{{\bf Dlg}}
\newcommand{\Fr}{{\bf Fr}}
\newcommand{\HFr}{{\bf HFr}}
\newcommand{\LCS}{\text{\bf LCS}}
\newcommand{\Con}{\text{\bf Con}}
\newcommand{\ConInf}{\Con{}^\infty}

\newcommand{\CSh}{{\mathfrak{C}\mathrm{Sh}}}
\newcommand{\CPre}{{\mathfrak{C}\mathrm{Pre}}}
\newcommand{\Pre}{{\mathfrak{P}\mathrm{re}}}
\newcommand{\Sh}{{\mathfrak{S}\mathrm{h}}}

\newcommand{\xyR}[1]{ \makeatletter
\xydef@\xymatrixrowsep@{#1} \makeatother} 
\newcommand{\xyC}[1]{ \makeatletter
\xydef@\xymatrixcolsep@{#1} \makeatother} 
\entrymodifiers={++[ ][F]} 
\newcommand{\OR}{{\mathcal{O}}\mbox{$\R^\infty$}}
\newcommand{\In}[1]{\in_{_{\scriptscriptstyle{#1}}}} % \in for generalized elements
\newcommand{\Man}{{\bf Man}}
\newcommand{\Set}{{\bf Set}}
\newcommand{\Top}{{\bf Top}}
\newcommand{\Fgen}{{\bf F}\Diff}
\newcommand{\TD}{\text{\rm T}_\text{\rm D}} % functor \Diff -> \Top "topological from diffeol"
\newcommand{\DT}{\text{\rm D}_\text{\rm T}} % functor \Top -> \Diff "diffeol from topological"
\newcommand{\DF}{\text{\rm D}_\text{\rm F}} % functor \Fr -> \Diff "diffeol from Frolicher"
\newcommand{\FD}{\text{\rm F}_\text{\rm D}} % functor \Diff -> \Fr "Frolicher from diffeol"
\newcommand{\FGF}{\text{\rm FG}_\text{\rm F}} % functor \Fr -> \Fgen "Fgen from Frolicher"
\newcommand{\DFG}{\text{\rm D}_\text{\rm FG}} % functor \Fgen -> \Diff "diffeol from Fgen"
\newcommand{\FGD}{\text{\rm FG}_\text{\rm D}} % functor \Diff -> \Fgen "Fgen from diffeol"
\newcommand{\FGM}{\text{\rm FG}_\text{\rm M}} % functor \Man -> \Fgen "Fgen from man"
\newcommand{\ext}[1]{{}^\bullet #1} % extension functor
\newcommand{\Fer}{\ext{\Diff}} %category of Fermat spaces
\newcommand{\FR}{{^\bullet\R}} % Fermat reals
\newcommand{\Rtil}{\widetilde \R} % real Colombeau generalized number
\newcommand{\otilc}{\widetilde \Omega_c}
\newcommand{\eps}{\varepsilon}
 
\renewcommand{\phi}{\varphi} 
\newcommand{\gs}{\mathcal{G}^s}
\newcommand{\Gs}{\boldsymbol{\mathcal{G}}^s}
\newcommand{\gf}{\mathcal{G}^e}
\newcommand{\Gf}{\boldsymbol{\mathcal{G}}^e}
\newcommand{\ems}{\mathcal{E}^s_M}
\newcommand{\Ems}{\boldsymbol{\mathcal{E}}^s_M}
\newcommand{\ns}{\mathcal{N}^s}
\newcommand{\emf}{\mathcal{E}^e_M}
\newcommand{\Emf}{\boldsymbol{\mathcal{E}}^e_M}
\newcommand{\nf}{\mathcal{N}^e}
\newcommand{\Sball}{B^{{\scriptscriptstyle \text{\rm S}}}}
\newcommand{\diff}[1]{\,\hbox{\rm d}#1}               % dt,dx,... for integrals
\newcommand{\Dsmooth}{\D^\text{\rm s}}
\newcommand{\Dlc}{\D^{{\scriptscriptstyle \text{\rm LC}}}}
\newcommand{\ptind}{\displaystyle \mathop {\ldots\ldots\,}} % marks with smth over. Usage: \ptind^{...}

\providecommand{\definitionname}{Definition}
\providecommand{\examplename}{Example}
\providecommand{\remarkname}{Remark}
\providecommand{\theoremname}{Theorem}
\providecommand{\corollaryname}{Corollary}

  \addto\captionsbritish{\renewcommand{\corollaryname}{Corollary}}
  \addto\captionsbritish{\renewcommand{\definitionname}{Definition}}
  \addto\captionsbritish{\renewcommand{\examplename}{Example}}
  \addto\captionsbritish{\renewcommand{\lemmaname}{Lemma}}
  \addto\captionsbritish{\renewcommand{\remarkname}{Remark}}
  \addto\captionsbritish{\renewcommand{\theoremname}{Theorem}}
  \addto\captionsenglish{\renewcommand{\corollaryname}{Corollary}}
  \addto\captionsenglish{\renewcommand{\definitionname}{Definition}}
  \addto\captionsenglish{\renewcommand{\examplename}{Example}}
  \addto\captionsenglish{\renewcommand{\lemmaname}{Lemma}}
  \addto\captionsenglish{\renewcommand{\remarkname}{Remark}}
  \addto\captionsenglish{\renewcommand{\theoremname}{Theorem}}
  \providecommand{\corollaryname}{Corollary}
  \providecommand{\definitionname}{Definition}
  \providecommand{\examplename}{Example}
  \providecommand{\lemmaname}{Lemma}
  \providecommand{\remarkname}{Remark}
\providecommand{\theoremname}{Theorem}

\makeatother

  \addto\captionsbritish{\renewcommand{\corollaryname}{Corollary}}
  \addto\captionsbritish{\renewcommand{\definitionname}{Definition}}
  \addto\captionsbritish{\renewcommand{\examplename}{Example}}
  \addto\captionsbritish{\renewcommand{\factname}{Fact}}
  \addto\captionsbritish{\renewcommand{\lemmaname}{Lemma}}
  \addto\captionsbritish{\renewcommand{\remarkname}{Remark}}
  \addto\captionsbritish{\renewcommand{\theoremname}{Theorem}}
  \addto\captionsenglish{\renewcommand{\corollaryname}{Corollary}}
  \addto\captionsenglish{\renewcommand{\definitionname}{Definition}}
  \addto\captionsenglish{\renewcommand{\examplename}{Example}}
  \addto\captionsenglish{\renewcommand{\factname}{Fact}}
  \addto\captionsenglish{\renewcommand{\lemmaname}{Lemma}}
  \addto\captionsenglish{\renewcommand{\remarkname}{Remark}}
  \addto\captionsenglish{\renewcommand{\theoremname}{Theorem}}
  \providecommand{\corollaryname}{Corollary}
  \providecommand{\definitionname}{Definition}
  \providecommand{\examplename}{Example}
  \providecommand{\factname}{Fact}
  \providecommand{\lemmaname}{Lemma}
  \providecommand{\remarkname}{Remark}
\providecommand{\theoremname}{Theorem}

\begin{document}

\title{Categorical frameworks for generalized functions}

\author{Paolo Giordano \and Enxin Wu}

\thanks{P.\ Giordano has been supported by grant P25116-N25 of the Austrian
Science Fund FWF}

\address{\textsc{University of Vienna, Austria}}

\email{paolo.giordano@univie.ac.at }

\thanks{E.\ Wu has been partially supported by grant P25311-N25 of the Austrian
Science Fund FWF}

\address{\textsc{University of Vienna, Austria}}

\email{enxin.wu@univie.ac.at}


\begin{abstract}
We tackle the problem of finding a suitable categorical framework
for generalized functions used in mathematical physics for linear
and non-linear PDEs. We are looking for a Cartesian closed category
which contains both Schwartz distributions and Colombeau generalized
functions as natural objects. We study Frölicher spaces, diffeological
spaces and functionally generated spaces as frameworks for generalized
functions. The latter are similar to Frölicher spaces, but starting
from locally defined functionals. Functionally generated spaces strictly
lie between Frölicher spaces and diffeological spaces, and they form
a complete and cocomplete Cartesian closed category. We deeply study
functionally generated spaces (and Frölicher spaces) as a framework
for Schwartz distributions, and prove that in the category of diffeological
spaces, both the special and the full Colombeau algebras are smooth
differential algebras, with a smooth embedding of Schwartz distributions
and smooth pointwise evaluations of Colombeau generalized functions.
\end{abstract}
\maketitle
\tableofcontents{}

\section{Introduction: finding a categorical framework for generalized functions}

The problem of considering (generalized) derivatives of locally integrable
functions arises frequently in Physics, e.g.\ in idealized models
like in shock Mechanics, material points Mechanics, charged particles
in Electrodynamics, gravitational waves in General Relativity, etc.\ (see
e.g.\ \cite{Col92,GKOS,MO01}). Therefore, the need to perform calculations
with discontinuous functions like one deals with smooth functions
motivated the introduction of generalized functions (GF) as objects
extending, in some sense, the notion of function. As such, generalized
functions find deep applications in solutions of singular differential
equations (\cite{Hor,MO92}) and are naturally framed in (several)
theories of infinite dimensional spaces, from locally convex vector
spaces (\cite{KN}) and convenient vector spaces (\cite{KM}) up to
diffeological (\cite{I,KR06}) and Frölicher spaces (\cite{FK}).

The foundation of a rigorous linear theory of generalized functions
has been pioneered by L.\ Schwartz with a deep use of locally convex
vector space theory (\cite{Sch,Hor}), but heuristic multiplications
of distributions early appeared e.g.\ in quantum electrodynamics,
elasticity, elastoplasticity, acoustics and other fields (\cite{Col92,MO01}).
Despite the impossibility of a straightforward extension of Schwartz
linear theory (\cite{Sch54}) to an algebra extending pointwise product
of continuous functions, the theory of Colombeau algebras (see e.g.\ \cite{C1,C2,Col92,MO92,GKOS,MO01})
permits to bypass this impossibility in a very simple way by considering
an algebra of generalized functions which extends the pointwise product
of smooth functions.

The main aim of the present work is to study different categories
as frameworks for generalized functions. In particular, we introduce
the category $\Fgen$ of \emph{functionally generated spaces}. This
category has very nice properties and strictly lies between Frölicher
and diffeological spaces.

We start by defining the algebras from Colombeau theory that we will
consider in this work. Henceforth, we will use the notations of \cite{GKOS,Hor}
for the well-known Schwartz distribution theory.

\subsection{\label{sub:Special-and-fullCA}The special and full Colombeau algebras}

\subsubsection*{The special Colombeau algebra}

In this section, we fix some basic notations and terminology from
Colombeau theory. For details we refer to \cite{GKOS}. We include
zero in the natural numbers $\N=\{0,1,2,\ldots\}$. Henceforth, $\Omega$
will always be an open subset of $\R^{n}$ and we denote by $I$ the
interval $(0,1]$. The (special) Colombeau algebra on $\Omega$ is
defined as the quotient $\gs(\Omega):=\mathcal{E}_{M}^{s}(\Omega)/\mathcal{N}^{s}(\Omega)$
of \emph{moderate nets} over \emph{negligible nets}, where the former
is 
\begin{multline*}
\ems(\Omega):=\{(u_{\eps})\in\Coo(\Omega)^{I}\mid\forall K\Subset\Omega\,\forall\alpha\in\N^{n}\,\exists N\in\N:\\
\sup_{x\in K}|\partial^{\alpha}u_{\eps}(x)|=O(\eps^{-N})\}
\end{multline*}
 and the latter is 
\begin{multline*}
\ns(\Omega):=\{(u_{\eps})\in\Coo(\Omega)^{I}\mid\forall K\Subset\Omega\,\forall\alpha\in\N^{n}\,\forall m\in\N:\\
\sup_{x\in K}|\partial^{\alpha}u_{\eps}(x)|=O(\eps^{m})\}.
\end{multline*}

Throughout this paper, every asymptotic relation is for $\eps\to0^{+}$.
Nets in $\ems(\Omega)$ are written as $(u_{\eps})$, and we use $u=[u_{\eps}]$
to denote the corresponding equivalence class in $\gs(\Omega)$. For
$(u_{\eps})\in\ns(\Omega)$ we also write $(u_{\eps})\sim0$. Then
$\Omega\mapsto\gs(\Omega)$ is a fine and supple sheaf of differential
algebras and there exist sheaf embeddings of the space of Schwartz
distributions $\D'$ into $\gs$ (cf.\ \cite{GKOS}). A very simple
way to embed $\D'$ into $\gs$ is given by the following result (\cite{SteVic09}):
\begin{thm}
\label{thm:magicMollifier}There exists a net $\left(\psi_{\eps}\right)\in\D(\R^{n})^{I}$
with the properties: 
\begin{enumerate}[leftmargin={*},label=(\roman*),align=left]
\item \label{enu:emb_supp}$\forall\eps\in I\,\forall x\in\supp(\psi_{\eps}):\ |x|<1$; 
\item \label{enu:emb_int1}$\int\psi_{\eps}=1\forall\eps\in I$, where the
implicit integration is over the whole $\R^{n}$; 
\item \label{enu:emb_moderate}$\forall\alpha\in\N^{n}\,\exists N\in\N:\ \sup_{x\in\R^{n}}\left|\partial^{\alpha}\psi_{\eps}(x)\right|=O(\eps^{-N})$; 
\item \label{enu:emb_moments}$\forall j\in\N\,\exists\eps_{0}\in I\,\forall\eps\in(0,\eps_{0}]\,\forall\alpha\in\N^{n}:\ 1\le|\alpha|\le j\Rightarrow\int x^{\alpha}\psi_{\eps}(x)\diff{x}=0$; 
\item \textup{\label{enu:emb_intAbs}$\forall\eta\in\R_{>0}\,\exists\eps_{0}\in I\,\forall\eps\in(0,\eps_{0}]:\ \int|\psi_{\eps}|\le1+\eta$.} 
\end{enumerate}

In particular, if we set 
\[
\eps\odot\psi_{\eps}:x\in\R^{n}\mapsto\frac{1}{\eps^{n}}\psi_{\eps}\left(\frac{x}{\eps}\right)\in\R\qquad\forall\eps\in I
\]
 
\begin{equation}
\iota_{\Omega}(u):=[u*(\eps\odot\psi_{\eps}|_{\Omega})]\qquad\forall u\in\D'(\Omega)\label{eq:embeddingSpecial}
\end{equation}
 then we have: 
\begin{enumerate}[leftmargin={*},label=(\roman*),align=left,start=6]
\item \label{enu:embLin}$\iota_{\Omega}:\D'(\Omega)\ra\gs(\Omega)$ is
a linear embedding; 
\item \label{enu:embDer}$\partial^{\alpha}(\iota_{\Omega}(u))=\iota_{\Omega}(D^{\alpha}u)$
for all $u\in\D'(\Omega)$ and all $\alpha\in\N^{n}$, where $D^{\alpha}$
and $\partial^{\alpha}$ are the $\alpha$-partial differential operators
on $\D'(\Omega)$ and $\gs(\Omega)$, respectively; 
\item \label{enu:embSmooth}$\iota_{\Omega}(f)=[f]$ for all $f\in\Coo(\Omega)$.
\end{enumerate}
\end{thm}
The ring of constants in $\gs$ is denoted by $\Rtil$
and is called the \emph{ring of Colombeau generalized numbers} (CGN).
It is an ordered ring with respect to the order defined by $[x_{\eps}]\le[y_{\eps}]$
iff $\exists[z_{\eps}]\in\Rtil$ such that $(z_{\eps})\sim0$ and
$x_{\eps}\le y_{\eps}+z_{\eps}$ for $\eps$ sufficiently small. Even
if this order is not total, we can still define the infimum $[x_{\eps}]\wedge[y_{\eps}]:=[\min(x_{\eps},y_{\eps})]$,
and analogously the supremum of two elements. More generally, the
space of generalized points in $\Omega$ is $\widetilde{\Omega}=\Omega_{M}/\sim$,
where $\Omega_{M}=\{(x_{\eps})\in\Omega^{I}\mid\exists N\in\N:|x_{\eps}|=O(\eps^{-N})\}$
is called the \emph{set of moderate nets} and $(x_{\eps})\sim(y_{\eps})$
if $|x_{\eps}-y_{\eps}|=O(\eps^{m})$ for every $m\in\N$. By $\mathcal{N}$
we will denote the set of all negligible nets of real numbers $(x_{\eps})\in\R^{I}$,
i.e., such that $(x_{\eps})\sim0$.

The space of compactly supported generalized points $\otilc$ is defined
by $\Omega_{c}/\!\sim$, where $\Omega_{c}:=\{(x_{\eps})\in\Omega^{I}\mid\exists K\Subset\Omega\,\exists\eps_{0}\,\forall\eps<\eps_{0}:\ x_{\eps}\in K\}$
and $\sim$ is the same equivalence relation as in the case of $\widetilde{\Omega}$.
Any Colombeau generalized function (CGF) $u\in\gs(\Omega)$ acts on
generalized points from $\otilc$ by $u(x):=[u_{\eps}(x_{\eps})]$
and is uniquely determined by its point values (in $\Rtil$) on compactly
supported generalized points (\cite{GKOS}), but not on standard points.
A CGF $[u_{\eps}]$ is called \emph{compactly-bounded} (c-bounded)
from $\Omega$ into $\Omega'$ if for any $K\Subset\Omega$ there
exists $K'\Subset\Omega'$ such that $u_{\eps}(K)\subseteq K'$ for
$\eps$ small. This type of CGF is closed with respect to composition.
Moreover, if $u\in\gs(\Omega)$ is c-bounded from $\Omega$ into $\Omega'$
and $v\in\gs(\Omega')$, then $[v_{\eps}\circ u_{\eps}]\in\gs(\Omega)$.
For $x,y\in\Rtil^{n}$ we will write $x\approx y$ if $x-y$ is infinitesimal,
i.e., if $|x-y|\le r$ for all $r\in\R_{>0}$.

Topological methods in Colombeau theory are usually based on the so-called
\emph{sharp topology} (see e.g.\ \cite{AJOS} and references therein),
which is the topology generated by the balls $\Sball_{\rho}(x)=\{y\in\Rtil^{n}\mid|y-x|<\rho\}$,
where $|-|$ is the natural extension of the Euclidean norm on $\Rtil^{n}$,
i.e., $|[x_{\eps}]|:=[|x_{\eps}|]\in\Rtil$, and $\rho\in\Rtil_{>0}$
is positive invertible. Henceforth, we will also use the notation
$\Rtil^{*}:=\{x\in\Rtil\mid x\text{ is invertible}\}$. Finally, Garetto
in \cite{Gtop,Gtop2} extended the above construction to arbitrary
locally convex spaces by functorially assigning a space of CGF $\gs_{E}$
to any given locally convex space $E$. The seminorms of $E$ can
then be used to define pseudovaluations which in turn induce a generalized
locally convex topology on the $\widetilde{\Co}$-module $\gs_{E}$,
again called \emph{sharp topology}.

\subsubsection*{The full Colombeau algebra}

Clearly, the embedding $\iota_{\Omega}$ defined in \eqref{eq:embeddingSpecial}
depends on the net of maps $(\psi_{\eps})\in\D(\R^{n})^{I}$ whose
existence is given by Thm.~\ref{thm:magicMollifier}. This shall
not be considered only in a negative way: e.g.\ it is not difficult
to choose $(\psi_{\eps})$ so that the embedding satisfies the properties
that $H(0)=\iota_{\R}(H)(0)=[\int_{-\infty}^{0}\psi{}_{\eps}]=0$
and $\delta(0)=\iota_{\R}(\delta)(0)=[\psi_{\eps}(0)]$ is an infinite
number of $\Rtil$ (here $H$ is the Heaviside function and $\delta$
is the Dirac delta function). These properties are informally used
in several applications.

The main idea of the full Colombeau algebra is to consider a different
set of indices, instead of $I=(0,1]$, so as to obtain an intrinsic
embedding.
\begin{defn}
\label{def:fullCA}
\begin{enumerate}[leftmargin={*},label=(\roman*),align=left]
\item $\mathcal{A}_{0}(\Omega):=\left\{ \phi\in\D(\Omega)\mid\int\phi=1\right\} $,
$\mathcal{A}_{0}:=\mathcal{A}_{0}(\R^{n})$;
\item $\mathcal{A}_{q}(\Omega):=\left\{ \phi\in\mathcal{A}_{0}(\Omega)\mid\forall\alpha\in\N^{n}:\ 1\le|\alpha|\le q\Rightarrow\int x^{\alpha}\phi(x)\diff{x}=0\right\} $; 
\item $\mathcal{A}_{q}:=\mathcal{A}_{q}(\R^{n})$; 
\item $U(\Omega):=\left\{ (\phi,x)\in\mathcal{A}_{0}\times\Omega\mid\supp(\phi)\subseteq\Omega-x\right\} $; 
\item We say that $R\in\mathcal{E}^{e}(\Omega)$ iff $R:U(\Omega)\ra\R$
and 
\[
\forall\phi\in\mathcal{A}_{0}:\ R(\phi,-)\text{ is smooth on }\Omega\cap\{x\in\R^{n}\mid\supp(\phi)\subseteq\Omega-x\};
\]

\item We say that $R\in\emf(\Omega)$ iff $R\in\mathcal{E}^{e}(\Omega)$
and 
\[
\forall K\Subset\Omega\,\forall\alpha\in\N^{n}\,\exists N\in\N\,\forall\phi\in\mathcal{A}_{N}:\ \sup_{x\in K}\left|\partial^{\alpha}R(\eps\odot\phi,x)\right|=O(\eps^{-N});
\]

\item We say that $R\in\nf(\Omega)$ iff $R\in\mathcal{E}^{e}(\Omega)$
and 
\[
\forall K\Subset\Omega\,\forall\alpha\in\N^{n}\,\forall m\in\N\,\exists q\in\N\,\forall\phi\in\mathcal{A}_{q}:\ \sup_{x\in K}\left|\partial^{\alpha}R(\eps\odot\phi,x)\right|=O(\eps^{m});
\]

\item $\gf(\Omega):=\emf(\Omega)/\nf(\Omega)$ is called the \emph{full
Colombeau algebra}; 
\item The above mentioned intrinsic embedding $\iota_{\Omega}:\D'(\Omega)\ra\gf(\Omega)$
is defined by $\left(\iota_{\Omega}u\right)(\phi,x):=\langle u,\phi(?-x)\rangle$.
It verifies properties like \ref{enu:embLin}, \ref{enu:embDer} and
\ref{enu:embSmooth} in Thm.~\ref{thm:magicMollifier}. 
\end{enumerate}
\end{defn}
For motivations and details, see \cite{GKOS}.

\section{Functionally generated diffeologies}

\subsection{Preliminaries on diffeological spaces and Frölicher spaces}

Both diffeological spaces and Frölicher spaces are generalizations
of smooth manifolds, introduced by J.M. Souriau and A. Frölicher,
respectively, in the 1980's. The smooth structure (called the \emph{diffeology})
on a diffeological space is defined by some testing functions from
all open subsets of all Euclidean spaces to the given set, subject
to a covering condition, a presheaf condition and a sheaf condition
(see Def. \ref{def:diffeologicalSpace}). A possible intuitive description
of this structure on a diffeological space $X$ is that a diffeology
is the specification not only of a particular family of smooth functions
(like charts on manifolds) but of all the possible smooth maps of
the type $d:U\longrightarrow X$ for all open subsets $U\subseteq\R^{n}$
and for all $n\in\N$. We can roughly say that we have to specify
what are smooth curves, surfaces, etc. on the space $X$.

On a Frölicher space $X$ we consider only $U=\R$, i.e., the smooth
structure on the space is given by a set of smooth curves; moreover,
these curves are determined by (and they determine) a given set of
functionals, i.e., of smooth functions of the type $l:X\longrightarrow\R$
(see Def. \ref{def:FrolicherSpace}). The category $\Fr$ of all Frölicher
spaces is a full subcategory of the category $\Diff$ of all diffeological
spaces.

In the following subsections, we are going to focus on a family of
diffeological spaces called \emph{functionally generated (diffeological)
spaces}, where the diffeological structure is determined by a given
family of locally defined smooth functionals. As we will see in the
present work, these spaces frequently appear in functional analysis,
strictly lie between diffeological spaces and Frölicher spaces, and
the category $\Fgen$ of all these spaces behaves nicely \textendash{}
it is complete, cocomplete and Cartesian closed.

To simplify the notation, we write $\OR$ for the category of open
sets in Euclidean spaces and ordinary smooth functions.
\begin{defn}
\label{def:diffeologicalSpace}A \emph{diffeological space} $X=(|X|,\D)$
is a set $|X|$ together with a specified family of functions 
\[
\D=\cup_{U\in\OR}\D_{U}\text{ with }\D_{U}\subseteq\Set(U,|X|)
\]
 such that for any $U,V\in\OR$, the following three axioms hold: 
\begin{enumerate}[leftmargin={*},label=(\roman*),align=left]
\item Every constant function $d:U\ra|X|$ is in $\D_{U}$ (Covering condition); 
\item $d\circ f\in\D_{V}$ for any $d:U\ra|X|\in\D_{U}$ and any $f\in\mC^{\infty}(V,U)$
(Presheaf condition); 
\item Let $d\in\Set(U,|X|)$, and let $\{U_{i}\}_{i\in I}$ be an open covering
of $U$. Then $d\in\D_{U}$ if $d|_{U_{i}}\in\D_{U_{i}}$ for each
$i\in I$ (Sheaf condition).
\end{enumerate}
\end{defn}
For a diffeological space $X=(|X|,\D)$, every element in $\D$ is
called a \emph{plot} of $X$. We write $d\In{U}X$ to denote that
$d\in\D_{U}$, which will also be called a \emph{figure} of type \emph{$U$}
of the space \emph{$X$.}
\begin{defn}
\label{def:diffeoMorph}A morphism (also called \emph{smooth map})
$f:X\ra Y$ between two diffeological spaces $X=(|X|,\D^{X})$ and
$Y=(|Y|,\D^{Y})$ is a function $|f|:|X|\ra|Y|$ such that $f\circ d\in\D_{U}^{Y}$
for any $d\in\D_{U}^{X}$ and $U\in\OR$.
\end{defn}
If we write $f(d):=f\circ d$, by the covering condition of Def. \ref{def:diffeologicalSpace}
we have a generalization of the usual evaluation; moreover, $f:X\ra Y$
is smooth if and only if for all $U\in\OR$ and $d\In{U}X$, we have
$f(d)\In{U}Y$, i.e., $f$ take figures of type $U$ on the domain
to figures of the same type in the codomain. Moreover, $X=Y$ as diffeological
spaces if and only if for all $d$ and $U$, $d\In{U}X$ if and only
if $d\In{U}Y$. These and several other generalization of set-theoretical
properties justify the use of the symbol $\In{U}$.

All diffeological spaces with smooth maps form a category, which will
be denoted by $\Diff$. Given two diffeological spaces $X$ and $Y$,
we write $\Coo(X,Y)$ for the set of all smooth maps $X\ra Y$.

\medskip{}

Here is a list of basic properties of diffeological spaces. We refer
readers to the standard textbook~\cite{I} for more details.
\begin{rem}
\label{rem:propDiffSpaces}
\begin{enumerate}[leftmargin={*},label=(\roman*),align=left]
\item By a smooth manifold, we always assume it is Hausdorff and finite-dimensional.
Every smooth manifold $M$ is automatically a diffeological space
$\textbf{M}=(M,\D)$ with $d\In{U}\textbf{M}$ if and only if $d:U\ra M$
is smooth in the usual sense. We call this $\D$ the \emph{standard
diffeology} on $M$, and without specification, we always assume a
smooth manifold with this diffeology when viewed as a diffeological
space. Moreover, given two smooth manifolds $M$ and $N$, $f:\textbf{M}\ra\textbf{N}$
is smooth if and only if $f:M\ra N$ is smooth in the usual sense.
In other words, the category $\Man$ of all smooth manifolds and smooth
maps is fully embedded in $\Diff$. This justifies our notation $\Coo(X,Y)$
for the hom-set $\Diff(X,Y)$. Limits of smooth manifolds that already
exist in $\Man$ are preserved by this embedding (see Thm.~\ref{thm:limitsPreserv}).
Generally speaking the same property does not hold for colimits of
smooth manifolds that already exist in $\Man$. 
\item \label{enu:discreteDiff}Given a set $X$, the set of all diffeologies
on $X$ forms a complete lattice. The smallest diffeology is called
the \emph{discrete diffeology}, which consists of all locally constant
functions, and the largest diffeology is called the \emph{indiscrete
diffeology}, which consists of all set functions. Let $A=(X,\D_{A})$
and $B=(X,\D_{B})$ be two diffeological spaces with the same underlying
set. We simply write $A\subseteq B$ iff $1_{X}:A\ra B$ is smooth,
i.e., iff $\D_{A}\subseteq\D_{B}$.

Therefore, given a family of functions $\left\{ \iota_{i}:|X_{i}|\ra Y\right\} _{i\in I}$
from the underlying sets of the diffeological spaces $X_{i}$ to a
fixed set $Y$, there exists a smallest diffeology on $Y$ making
all these maps $\iota_{i}$ smooth. We call this diffeology the \emph{final
diffeology associated to $I$}. In more detail, 
\[
d\In{U}Y\text{ iff }\forall u\in U\,\exists V\text{ neigh. of }u\,\exists i\in I\,\exists\delta\In{V}X_{i}:\ \iota_{i}\circ\delta=d|_{V}.
\]
 Dually, given a family of functions $\left\{ p_{j}:X\ra|Y_{j}|\right\} _{j\in J}$
from a given set $X$ to the underlying sets of the diffeological
spaces $Y_{j}$, there exists a largest diffeology on $X$ making
all these maps $p_{j}$ smooth. We call this diffeology the \emph{initial
diffeology associated to $J$}. In more detail, 
\[
d\In{U}X\text{ iff }p_{j}\circ d\In{U}Y_{j}\ \forall j\in J.
\]

In particular, if $Y$ is a quotient set of $|X|$, then the final
diffeology on $Y$ associated to the quotient map $|X|\ra Y$ is called
the \emph{quotient diffeology}, and $Y$ with the quotient diffeology
is called a \emph{quotient diffeological space} of $X$. Dually, if
$X$ is a subset of $|Y|$, then the initial diffeology on $X$ associated
to the inclusion map $X\ra|Y|$ is called the \emph{sub-diffeology},
and we write $(X\prec Y)$ to denote this new diffeological space.
We call $(X\prec Y)$ the \emph{diffeological subspace} of $Y$. Finally,
the initial diffeology associated to the projection maps $p_{i}:\prod_{i\in I}|X_{i}|\ra|X_{i}|$
of an arbitrary product is called the \emph{product diffeology}, and
dually the final diffeology associated to the inclusion maps $|X_{j}|\ra\coprod_{j\in J}|X_{j}|$
of an arbitrary coproduct is called the \emph{coproduct diffeology}.

\item \label{enu:limitsColimitsDiff}The category $\Diff$ is complete and
cocomplete. In more detail, let $G:\mathcal{I}\ra\Diff$ be a functor
from a small category $\mathcal{I}$. Write $|-|:\Diff\ra\Set$ for
the forgetful functor. Then both $\lim G$ and $\colim G$ exist in
$\Diff$ as lifting and co-lifting of limits and colimits in $\Set$.
In more detail, $|\lim G|=\lim|G|$ and the diffeology of $\lim G$
is the initial diffeology associated to the universal cone $\{\lim|G|\ra|G(i)|\}_{i\in\mathcal{I}}$
in $\Set$; dually $|\colim G|=\colim|G|$ and the diffeology of $\colim G$
is the final diffeology associated to the universal co-cone $\{|G(i)|\ra\colim|G|\}_{i\in\mathcal{I}}$
in $\Set$.
\item The category $\Diff$ is Cartesian closed. In more detail, given three
diffeological spaces $X$, $Y$ and $Z$, there is a canonical diffeology
(called the \emph{functional diffeology}) on $\Coo(X,Y)$ defined
by 
\[
d\In{U}\Coo(X,Y)\text{ iff }d^{\vee}\in\Coo(U\times X,Y),
\]
 with $d^{\vee}(u,x)=d(u)(x)$ (in the present work, we use the notations
of\ \cite{AHS}). Without specification, the set $\Coo(X,Y)$ is
always equipped with the functional diffeology when viewed as a diffeological
space. Then Cartesian closedness means that $f\in\Coo(X,\Coo(Y,Z))$
if and only if $f^{\vee}\in\Coo(X\times Y,Z)$ (or, equivalently that
$g\in\Coo(X\times Y,Z)$ if and only if $g^{\wedge}\in\Coo(X,\Coo(Y,Z))$,
where $g^{\wedge}(x)(y):=g(x,y)$). Therefore, Cartesian closedness
permits to equivalently translate an infinite dimensional problem
like $f\in\Coo(X,\Coo(Y,Z))$ into a finite dimensional one $f^{\vee}\in\Coo(X\times Y,Z)$
and vice versa.
\item Every diffeological space can be extended with infinitely near points
$X\in\Diff\mapsto\ext{X}\in\Fer$, $X\subseteq\ext{X}$, obtaining
a nonArchimedean framework similar to Synthetic Differential Geometry
(see e.g.\ \cite{MoeRey} and references therein) but compatible
with classical logic. The category $\Fer$ of \emph{Fermat spaces}
is defined by generalizing the category of diffeological spaces, but
taking suitable smooth functions defined on the extension $\ext{U}\subseteq\FR^{n}$
of open sets $U\in\OR$. It is remarkable to note that the so called
\emph{Fermat functor} $\ext{(-)}:\Diff\ra\Fer$ has very good preservation
properties strictly related to intuitionistic logic. See \cite{Gio10a,Gio10b,Gio10c,Gio10d,GK}
for more details. 
\item $\Diff$ is a quasi-topos and hence is locally Cartesian closed (\cite{Ba-Ho11}). 
\end{enumerate}
\end{rem}
Every diffeological space has an interesting canonical topology:
\begin{defn}
Let $X=(|X|,\D)$ be a diffeological space. The final topology $\tau_{X}$
induced by $\D$ is called the \emph{$D$-topology}.
\end{defn}
Without specification, every diffeological space $X$ is equipped
with the $D$-topology $\tau_{X}$. Elements in $\tau_{X}$ are called
\emph{$D$-open subsets}.
\begin{example}
\ 
\begin{enumerate}[leftmargin={*},label=(\roman*),align=left]
\item The $D$-topology on any smooth manifold is the usual topology. 
\item The $D$-topology on any discrete (indiscrete) diffeological space
is the discrete (indiscrete) topology. 
\end{enumerate}
\end{example}
\begin{thm}
\cite{SYH} $\TD:\Diff\ra\Top$ defined by $\TD(X)=(|X|,\tau_{X})$
is a functor%
\footnote{We can recall the symbol $\TD$ by saying ``topological space from
diffeological space''. Analogously we can recall the plenty of symbols
for the other functors related to our categories in this paper.%
}, which has a right adjoint $\DT:\Top\ra\Diff$ defined by $|\DT(X)|:=|X|$
and $d\In{U}\DT(X)$ iff $d\in\Top(U,X)$ (both functors act as identity
on arrows).
\end{thm}
As a consequence, the $D$-topology of a quotient diffeological space
of $X$ is same as the quotient topology of $\TD(X)$. However, the
$D$-topology of a diffeological subspace of $X$ may be different
from the sub-topology of $\TD(X)$.

For more detailed discussion of the $D$-topology of diffeological
spaces, see~\cite[Chapter~2]{I} and~\cite{CSW}.

\bigskip{}

Now, let's turn to Frölicher spaces. In several spaces of functional
analysis (like all those listed in Section \ref{sub:Categorical-frameworks}),
smooth figures are ``generated by smooth functionals''. Therefore,
smoothness can also be tested using smooth functionals, similarly
as using projections in finite dimensional Euclidean spaces. In Frölicher
spaces, we focus our attention also to smooth functions of the type
$X\ra\R$.
\begin{defn}
\label{def:FrolicherSpace}A Frölicher space $(\mathcal{C},X,\F)$
is a set $X$ together with two specified families of functions 
\[
\mathcal{C}\subseteq\Set(\R,X)\text{ and }\F\subseteq\Set(X,\R)
\]
 with the following smooth compatibility conditions: 
\[
c:\R\ra X\in\mathcal{C}\text{ iff }l\circ c\in\Coo(\R,\R)\ \forall l\in\F,
\]
 and 
\[
f:X\ra\R\in\F\text{ iff }l\circ c\in\Coo(\R,\R)\ \forall c\in\mathcal{C}.
\]

\end{defn}
\medskip{}

\begin{defn}
A morphism $f:(\mathcal{C}_{X},X,\F_{X})\ra(\mathcal{C}_{Y},Y,\F_{Y})$
between two Frölicher spaces is a function $f:X\ra Y$ such that one
of the following equivalent conditions hold: 
\begin{enumerate}[leftmargin={*},label=(\roman*),align=left]
\item $f\circ c\in\mathcal{C}_{Y}\ \forall c\in\mathcal{C}_{X}$, 
\item $l\circ f\in\F_{X}\ \forall l\in\F_{Y}$, 
\item $l\circ f\circ c\in\Coo(\R,\R)\ \forall c\in\mathcal{C}_{X}\text{ and }\forall l\in\F_{Y}$.
\end{enumerate}
\end{defn}
All Frölicher spaces and their morphisms form a category, which will
be denoted by $\Fr$. Here is a list of basic properties for Frölicher
spaces. For details, we refer readers to \cite{FK}.
\begin{rem}
\label{rem:propFrSpaces}
\begin{enumerate}[leftmargin={*},label=(\roman*),align=left]
\item Every smooth manifold $M$ is automatically a Frölicher space with
$\mathcal{C}=\Coo(\R,M)$ and $\F=\Coo(M,\R)$. Without specification,
we always assume a smooth manifold with this Frölicher structure when
viewed as a Frölicher space. Moreover, this gives a full embedding
of $\Man$ in $\Fr$.
\item Let $\{\iota_{i}:X_{i}\ra Y\}_{i\in I}$ be a family of functions
from the underlying sets of the Frölicher spaces $(\mathcal{C}_{i},X_{i},\F_{i})$
to a fixed set $Y$. Let $\F_{Y}=\{l:Y\ra\R\mid l\circ\iota_{i}\in\F_{i}\ \forall i\}$
and let $\mathcal{C}_{Y}=\{c:\R\ra Y\mid l\circ c\in\Coo(\R,\R)\ \forall l\in\F_{Y}\}$.
Then $(\mathcal{C}_{Y},Y,\F_{Y})$ is a Frölicher space and all these
maps $\iota_{i}$ are morphisms between Frölicher spaces. We call
this Frölicher structure on $Y$ the \emph{final Frölicher structure
associated to }$\{\iota_{i}:X_{i}\ra Y\}_{i\in I}$.\\
 Dually, let $\{p_{j}:X\ra Y_{j}\}_{j\in J}$ be a family of functions
from a fixed set $X$ to the underlying sets of the Frölicher spaces
$(\mathcal{C}_{j},Y_{j},\F_{j})$. Let $\mathcal{C}_{X}=\{c:\R\ra X\mid p_{j}\circ c\in\mathcal{C}_{j}\ \forall j\}$
and let $\F_{X}=\{l:X\ra\R\mid l\circ c\in\Coo(\R,\R)\ \forall c\in\mathcal{C}_{X}\}$.
Then $(\mathcal{C}_{X},X,\F_{X})$ is a Frölicher space and all these
maps $p_{j}$ are morphisms between Frölicher spaces. We call this
Frölicher structure on $X$ the \emph{initial Frölicher structure
associated to }$\{p_{j}:X\ra Y_{j}\}_{j\in J}$. 
\item The category $\Fr$ is complete and cocomplete. In more detail, let
$G:\mathcal{I}\ra\Fr$ be a functor from a small category $\mathcal{I}$.
Write $|-|:\Fr\ra\Set$ for the forgetful functor. Then both $\lim G$
and $\colim G$ exist in $\Fr$ as lifting and co-lifting of limits
and colimits in $\Set$. In more detail, $|\lim G|=\lim|G|$ and the
Frölicher structure of $\lim G$ is the initial Frölicher structure
associated to the universal cone $\{\lim|G|\ra|G(i)|\}_{i\in\mathcal{I}}$
in $\Set$; dually $|\colim G|=\colim|G|$ and the Frölicher structure
of $\colim G$ is the final Frölicher structure associated to the
universal co-cone $\{|G(i)|\ra\colim|G|\}_{i\in\mathcal{I}}$. In
the category of Hausdorff Frölicher spaces, limits and colimits of
smooth manifolds that already exist in $\Man$ are preserved by the
embedding $\Man\ra\HFr$ (see Thm.~\ref{thm:preservecolimits}). 
\item The category $\Fr$ is Cartesian closed. In more detail, given Frölicher
spaces $X$ and $Y$, set $\mC=\left\{ c:\R\ra\Fr(X,Y)\mid c^{\vee}\in\Fr(\R\times X,Y)\right\} $,
and $\F=\{l:\Fr(X,Y)\ra\R\mid l\circ c\in\Coo(\R,\R)\ \forall c\in\mathcal{C}\}$.
Then one can show that $(\mathcal{C},\Fr(X,Y),\F)$ is a Frölicher
space. Without specification, $\Fr(X,Y)$ is always equipped with
this Frölicher structure when viewed as a Frölicher space. 
\item \label{enu:functorsFrDiff}Given a Frölicher space $(\mathcal{C},Y,\F)$,
let $\D_{U}=\{d:U\ra Y\mid l\circ d\in\Coo(U,\R)\ \forall l\in\F\}$.
Then $\DF(\mC,Y,\F):=(Y,\D=\cup_{U\in\OR}\D_{U})$ is a diffeological
space. This defines a full embedding $\DF:\Fr\ra\Diff$. So, there
will be no confusion to call smooth maps also the morphisms between
Frölicher spaces. Moreover, one can show that $\DF(\Fr(A,B))=\Coo(\DF(A),\DF(B))$
as diffeological spaces. This embedding functor has a left adjoint
given as follows. For a diffeological space $X=(|X|,\D)$, let $\F=\Coo(X,\R)$
and let $\mathcal{C}=\{c:\R\ra X\mid l\circ c\in\Coo(\R,\R)\ \forall l\in\F\}$.
Then $\FD(X):=(\mathcal{C},|X|,\F)$ is a Frölicher space. Both functors
$\DF$ and $\FD$ are identities on the morphisms. For more discussion
on the relationship between diffeological spaces and Frölicher spaces,
see~\cite{S,BIKW}.
\end{enumerate}
\end{rem}

\subsection{Definition and examples of functionally generated diffeologies}

Now, let's introduce a special class of diffeological spaces called
\emph{functionally generated spaces}, which are like Frölicher spaces
but with locally defined smooth functionals. The idea is that, in
this type of spaces, we can know whether a continuous function $d\in\Top(U,\TD(X))$
is a figure by testing for smoothness of the composition with a given
family of locally defined smooth functionals $l:(A\prec X)\ra\R$.
\begin{defn}
\label{def:functionallyGeneratedSpace} Let $X=(|X|,\D)$ be a diffeological
space. Let $\F=\left\{ \F_{A}\right\} _{A\in\tau_{X}}$ be a $\tau_{X}$-family
of smooth functions, that is, for each $A\in\tau_{X}$ 
\[
\F_{A}\subseteq\Coo(A\prec X,\R).
\]
 We say that $\F$ \emph{generates} $\D$ if for any open set $U\in\OR$
and any continuous map $d\in\Top(U,\TD(X))$, the condition 
\begin{equation}
\forall A\in\tau_{X}\,\forall l\in\F_{A}:\ l\circ d|_{d^{-1}(A)}\in\Coo(d^{-1}(A),\R)\label{eq:compFunctionals}
\end{equation}

\noindent implies $d\In{U}X$, i.e., $d$ is a plot of $X$. Any map
$l\in\Coo(A\prec X,\R)$ is called a \emph{smooth functional} of the
space $X$. Finally, we say that the diffeological space $X$ is \emph{functionally
generated} if its diffeology can be generated by some family $\F$,
and we denote with $\Fgen$ the full subcategory of $\Diff$ of all
functionally generated diffeological spaces.
\end{defn}
\noindent If the codomain of a continuous map $f:X\ra Y$ is functionally
generated, then we can also test the smoothness of $f$ by smooth
functionals of $Y$:
\begin{thm}
\noindent \label{thm:codomFgenTestFunctionals}Let $f:|X|\ra|Y|$
be a map with $X\in\Diff$ and $Y\in\Fgen$. Assume that the diffeology
of $Y$ is generated by the family $\left\{ \F_{A}\right\} _{A\in\tau_{Y}}$.
Then the following are equivalent 
\begin{enumerate}[leftmargin={*},label=(\roman*),align=left]
\item \label{enu:codomFgenTestFunct_old}$f\in\Coo(X,Y)$ 
\item \label{enu:codomFgenTestFunct_new}$f\in\Top(\TD(X),\TD(Y))$ and
\begin{equation}
\forall A\in\tau_{Y}\,\forall l\in\F_{A}:\ l\circ f|_{f^{-1}(A)}\in\Coo(f^{-1}(A)\prec X,\R).\label{eq:codomFgenTestFunct_cond}
\end{equation}

\end{enumerate}
\end{thm}
\begin{proof}
Since the implication \ref{enu:codomFgenTestFunct_old} $\Rightarrow$
\ref{enu:codomFgenTestFunct_new} is clear, we only prove the opposite
one. For any $d\In{U}X$, since $d\in\Top(U,\TD(X))$, $f\circ d\in\Top(U,\TD(Y))$.
Then for any $A\in\tau_{Y}$ and any $l\in\F_{A}$, by \eqref{eq:codomFgenTestFunct_cond}
we get $l\circ f|_{f^{-1}(A)}\in\Coo(f^{-1}(A)\prec X,\R)$ and hence
$l\circ f|_{f^{-1}(A)}\circ d|_{d^{-1}(f^{-1}(A))}=l\circ\left(f\circ d\right)|_{(f\circ d)^{-1}(A)}\in\Coo((f\circ d)^{-1}(A),\R)$.
Since the diffeology of $Y$ is generated by $\left\{ \F_{A}\right\} _{A\in\tau_{Y}}$,
the conclusion $f\circ d\In{U}Y$ follows.
\end{proof}
Here is a list of basic properties and examples of functionally generated
spaces:
\begin{rem}
\label{rem:equiv}\ 
\begin{enumerate}[leftmargin={*},label=(\roman*),align=left]
\item \label{enu:local}The notion of functionally generated space is of
local nature, i.e., we can equivalently say that $\F$ generates $\D$
if for any $U\in\OR$ and any $d\in\Set(U,|X|)$, the condition 
\begin{multline*}
\hspace{10mm}\forall u\in U\,\forall A\in\tau_{X}\,\forall l\in\F_{A}:\ d(u)\in A\Rightarrow\\
\exists V\text{ neigh. of }u:\ d(V)\subseteq A\ ,\ l\circ d|_{V}\in\Coo(V,\R)
\end{multline*}
 implies $d\In{U}X$. 
\item We can also equivalently ask that $\F_{A}\subseteq\Set(A,\R)$ and
for all continuous $d\in\Top(U,\TD(X))$ and any open set $U\in\OR$
we have $d\in_{U}X$ if and only if \eqref{eq:compFunctionals} holds.
Therefore, smooth functionals determine completely the figures (plots)
of the underlying diffeological space, i.e., if $\D_{1}$, $\D_{2}$
are diffeologies on $|X|$ and $\F$ generates both $\D_{1}$ and
$\D_{2}$, then $\D_{1}=\D_{2}$. 
\item \label{enu:FMax}Let $\F$ generate $\D$. Define $\M_{A}^{X}:=\Coo(A\prec X,\R)$
for any $A\in\tau_{X}$. Then $\M^{X}$ also generates $\D$. Of course,
$\M^{X}$ is the maximum family of smooth functionals which can be
used to test whether a continuous map $d\in\Top(U,\TD(X))$ is a figure
or not, and the interesting problem is to find a smaller family $\F\subseteq\M^{X}$
generating the same set of plots of $X$.
\item \label{enu:FrisFgen}The diffeology generated by a Frölicher space
$(\mathcal{C},X,\F)$ is functionally generated by globally defined
smooth functionals. That is, it suffices to consider $\tilde{\F}$
defined by $l\in\tilde{\F}_{A}$ if and only if $A=X$ and $l\in\F$.
Therefore, the functor $\FD:\Fr\ra\Diff$ has values in $\Fgen$.
In particular, every smooth manifold and every discrete diffeological
space is functionally generated. However, there are functionally generated
spaces which do not come from Frölicher spaces; see Ex.~\ref{eg:nonfgen}.
\end{enumerate}
\end{rem}
In a functionally generated space, besides the usual $D$-topology
$\tau_{X}$, we can consider the initial topology $\tau_{\F}$ with
respect to all smooth functionals $\bigcup_{A\in\tau_{X}}\F_{A}$
(which is analogous to the weak topology, see e.g.\ \cite{KN}).
In particular, the topology $\tau_{\M^{X}}$ is called the \emph{functional
topology} on \emph{$X$. }In general, $\tau_{\F}$ is coarser than
the $D$-topology, see Ex.~\ref{eg:nonfgen}, but in every functionally
generated space the functional topology and the $D$-topology coincide,
as stated in the following theorem. 
\begin{thm}
\label{thm:functionaltopology} Let $\F$ generate the diffeology
of the space $X\in\Diff$. Then $\tau_{\F}\subseteq\tau_{X}$. If
$\F_{A}\neq\emptyset$ for all $A\in\tau_{X}$, then $\tau_{\F}=\tau_{X}$.
In particular, $\tau_{\M^{X}}=\tau_{X}$.\end{thm}
\begin{proof}
The topology $\tau_{\F}$ has the set 
\[
\{l^{-1}(V)\mid l\in\F_{A},V\in\tau_{\R},A\in\tau_{X}\}
\]
 as a subbase. For any $l^{-1}(V)$ in this subbase and any plot $d\In{U}X$
the set 
\[
d^{-1}(l^{-1}(V))=(l\circ d|_{d^{-1}(A)})^{-1}(V)
\]
 is open in $U$ since $l\circ d|_{d^{-1}(A)}$ is smooth by Def.
\ref{def:functionallyGeneratedSpace}. Hence $\tau_{\F}\subseteq\tau_{X}$.
Vice versa, if $A\in\tau_{X}$ and $l\in\F_{A}\ne\emptyset$ is any
functional, then $A=l^{-1}(\R)\in\tau_{\F}$. So $\tau_{X}\subseteq\tau_{\F}$
if $\F_{A}\neq\emptyset$ for all $A\in\tau_{X}$.
\end{proof}
Here are some examples of diffeological spaces which are not functionally
generated.
\begin{example}
\label{ex:nonfgen}\ 
\begin{enumerate}[leftmargin={*},label=(\roman*),align=left]
\item \label{ex:irrtorus} Let $(X,\D)$ be the irrational torus $\R/(\Z+\theta\Z)$,
for some $\theta\in\R\setminus\Q$, with the quotient diffeology $\D$;
see \cite{I}. Then $(X,\D)$ is not functionally generated. Indeed,
since $\Z+\theta\Z$ is dense in $\R$, the $D$-topology $\tau_{X}$
is indiscrete, that is, $\tau_{X}=\{\emptyset,X\}$. Hence, every
smooth map $X\ra\R$ is constant. Therefore, for any function $d:U\ra X$,
for any $l\in\M^{X}$, the composition $l\circ d$ is constant, hence
smooth. Therefore, there does not exist a family $\F$ that generates
$\D$. 
\item For any $n\geq2$, let $\R_{\text{w}}^{n}=(\R^{n},\D^{\text{w}})$
be $\R^{n}$ with the wire diffeology $\D^{\text{w}}$; see \cite{I}.
Then the $D$-topology of $\R_{\text{w}}^{n}$ is the usual Euclidean
topology, and by Boman's theorem \cite{Bom}, $\Coo(A\prec\R_{\text{w}}^{n},\R)=\Coo(A\prec\R^{n},\R)=\M_{A}^{\R_{\text{w}}^{n}}$
so that if $\R_{\text{w}}^{n}$ is functionally generated, we would
have $1_{\R^{n}}\In{\R^{n}}\R_{\text{w}}^{n}$, which is false for
the wire diffeology. Therefore, $\R_{\text{w}}^{n}$ is not functionally
generated.
\end{enumerate}
\end{example}

\subsection{Categorical properties of functionally generated spaces}

In this subsection, we are going to prove some nice categorical properties
for the category $\Fgen$ of all functionally generated spaces with
smooth maps, that is, $\Fgen$ is complete, cocomplete and Cartesian
closed.

\medskip{}

Although the family $\F$ that generates a diffeology is a $\tau_{X}$-family
of smooth functions, in practice, we usually only need a $\mathcal{B}$-family
with $\mathcal{B}\subseteq\tau_{X}$. In other words, $\F_{A}$ can
be any subset (in particular, the empty set) of $\Coo(A\prec X,\R)$
if $A\in\tau_{X}\setminus\mathcal{B}$. We have already met such examples
in~\ref{enu:FrisFgen} of Rem. \ref{rem:equiv}. Here is another
big class of examples:
\begin{thm}
\label{thm:initial}Let $\{p_{i}:X\ra|X_{i}|\}_{i\in I}$ be a family
of functions from a given set $X$ to the underlying sets of the diffeological
spaces $X_{i}=(|X_{i}|,\mathcal{D}^{i})$. Assume that each $\D^{i}$
is generated by $\F^{i}$, and let $\D$ be the initial diffeology
on $X$ associated to this family (i.e., $d\in_{U}\D\text{ iff }p_{i}\circ d\in_{U}\D^{i}\ \forall i$).
For all $A\in\tau_{(X,\D)}$ set $l\in\F_{A}$ iff $l\in\Coo(A\prec X,\R)$
and 
\[
\exists i\in I\,\exists B\in\tau_{X_{i}}\,\exists\lambda\in\F_{B}^{i}:\ A=p_{i}^{-1}(B)\ ,\ l=\lambda\circ p_{i}|_{A}.
\]
 Then the diffeology $\D$ is generated by $\F$.
\end{thm}
The proof follows directly from Def. \ref{def:functionallyGeneratedSpace}.
Note that $\F_{A}=\emptyset$ if $A$ is not of the form $A=p_{i}^{-1}(B)$
for some $i\in I$ and $B\in\tau_{X_{i}}$, so that we are essentially
considering only smooth functionals defined on $D$-open subsets in
$\mathcal{B}=\{p_{i}^{-1}(B)\mid B\in\tau_{X_{i}},i\in I\}\subseteq\tau_{(X,\D)}$.
In this sense, $\F$ is also the smallest family of smooth functionals
generating $(X,\D)$ and containing all the smooth functionals of
the form $\lambda\circ p_{i}|_{p_{i}^{-1}(B)}$.

\noindent In particular, every subset of a functionally generated
space with the sub-diffeology is again functionally generated. Analogously,
every product of functionally generated spaces with the product diffeology
is functionally generated, so that
\begin{cor}
\noindent \label{cor:FgenComplete}The category $\Fgen$ is complete.
\end{cor}
\noindent Similarly, one can show that every coproduct of functionally
generated spaces with the coproduct diffeology is functionally generated.

Let $f,g:X\ra Y$ be smooth maps between functionally generated spaces.
In general, the coequalizer in $\Diff$ may not be functionally generated.
For example, let $X=\R$ be equipped with the discrete diffeology,
and let $Y=\R$ be equipped with the standard diffeology. Let $\theta$
be some irrational number. Fix a representative in $\R$ for each
element in $\R/(\Z+\theta\Z)$, i.e., a function $\rho:\R/(\Z+\theta\Z)\ra\R$
such that $\rho(c)\in c$ for all $c\in\R/(\Z+\theta\Z)$. Let $f:X\ra Y$
be the identity function and let $g:X\ra Y$ be the function defined
by $g(r):=\rho(c)$ for all $r\in c\in\R/(\Z+\theta\Z)$, i.e., sending
every point in the subset $\rho(c)+\Z+\theta\Z$ to the fixed representative
$\rho(c)\in\R$. It is clear that both $f$ and $g$ are smooth because
we equip $X$ with locally constant figures, and the coequalizer in
$\Diff$ is the irrational torus because the equivalence relation
of $\R/(\Z+\theta\Z)$ is the smallest one where $f(r)=r$ is equivalent
to $g(r)=\rho(c)$ for $r\in c$. We already know from~\ref{ex:irrtorus}
of Ex.~\ref{ex:nonfgen} that the irrational torus is not functionally
generated. However, we will show below that the category $\Fgen$
is cocomplete, and the coequalizer of the above diagram in $\Fgen$
is the underlying set of the irrational torus with the indiscrete
diffeology.

\medskip{}

Now, we want to see how to define a new functionally generated diffeology
starting from a diffeological space and a $\tau_{X}$-family of smooth
functionals.
\begin{defn}
\label{def:DF}Let $X=(|X|,\sigma)\in\Top$ be a topological space,
and let $\F=\left\{ \F_{B}\right\} _{B\in\mathcal{B}}$ with $\mathcal{B}\subseteq\sigma$
be a $\mathcal{B}$-family of functions, i.e., $\F_{B}\subseteq\Set(B,\R)$.
For $U\in\OR$, write $d\in\D\F_{U}$ (or $\D^{X}\F_{U}$ if we need
to show the dependence from $X$) if and only if $d\in\Top(U,X)$
and 
\[
\forall B\in\mathcal{B}\,\forall l\in\F_{B}:\ l\circ d|_{d^{-1}(B)}\in\Coo(d^{-1}(B),\R).
\]
 We set $\D\F:=\cup_{U\in\OR}\D\F_{U}$ and call $\hat{X}_{\F}:=(|X|,\D\F)$
the \emph{diffeological space generated by $X$ and $\F$}. If $Y\in\Diff$
is a diffeological space, we will always apply the above construction
with respect to the $D$-topology, i.e., with $X=\TD(Y)$ and considering
only smooth functions: $\F_{B}\subseteq\Coo(B\prec X,\R)$ for all
$B\in\mathcal{B}$.
\end{defn}
One can show directly from the definitions that
\begin{rem}
\label{rem:propDF}In the hypotheses of the previous Def. \ref{def:DF},
the following properties hold: 
\begin{enumerate}[leftmargin={*},label=(\roman*),align=left]
\item We can trivially extend the $\mathcal{B}$-family $\F$ to the whole
$\sigma$-family by setting $\F_{A}:=\emptyset$ if $A\notin\mathcal{B}$.
We will always assume to have extended $\F$ in this way; 
\item $\D\F$ is a diffeology on $|X|$; 
\item \label{enu:inclusion} For all $A\in\sigma$ and $l\in\F_{A}$, we
have $l\in\Coo(A\prec\hat{X}_{\F},\R)$; 
\item The diffeology $\D\F$ of $\hat{X}_{\F}$ is functionally generated
by $\F$. 
\end{enumerate}

Moreover, if $X=(|X|,\D)\in\Diff$ is a diffeological space, then 
\begin{enumerate}[leftmargin={*},label=(\roman*),align=left,start=5]
\item \label{enu:ext}$\D_{U}\subseteq\D\F_{U}$. Hence, $\Coo(Y\prec X,\R)\supseteq\Coo(Y\prec\hat{X}_{\M^{X}},\R)\,\forall Y\subseteq|X|$.
Together with the above Property \ref{enu:inclusion}, we have $\Coo(A\prec X,\R)=\Coo(A\prec\hat{X}_{\M^{X}},\R)\,\forall A\in\tau_{X}$; 
\item \label{enu:FgeneratesDequiv}$\F$ generates $\D$ if and only if
$\D_{U}=\D\F_{U}$; 
\item \label{enu:topOfFunctExtension}the $D$-topology on $\hat{X}_{\F}$
coincides with the $D$-topology $\tau_{X}$ on $X$; 
\item \label{enu:XisXhatF}If $\F$ generates $\D$, then $X=\hat{X}_{\F}$.
So $\widehat{\hat{X}_{\F}}_{\F}=\hat{X}_{\F}$ for any $\tau_{X}$-family
$\F$. And if $X$ is functionally generated, then $X=\hat{X}_{\M^{X}}$; 
\item $\D\M^{X}$ is the smallest functionally generated diffeology on $|X|$
containing $\D$. 
\end{enumerate}
\end{rem}
In particular, if we take $\F$ to be the empty $\tau_{X}$-family,
that is, $\F_{A}=\emptyset$ for all $A\in\tau_{X}$, then $\D\F_{U}=\Top(U,T_D(X))$.
\begin{thm}
\label{thm:fgtfulInDiffIsRightAdjoint}The inclusion functor $\xymatrix{\DFG:\Fgen\ar@{^{(}->}[r] & \Diff}
$ is a right adjoint of the functor $\FGD:X\in\Diff\mapsto\hat{X}_{\M^{X}}\in\Fgen$
(both functors act as identity on arrows). Therefore, for all $X\in\Diff$
and $Y\in\Fgen$, we have 
\[
\Coo(X,Y)=\Coo\left(\hat{X}_{\M^{X}},Y\right).
\]
 We call $\FGD(X)=(|X|,\D^{X}\M^{X})$ the \emph{functional extension}
of $X$.\end{thm}
\begin{proof}
It follows by applying Def. \ref{def:functionallyGeneratedSpace},
Def. \ref{def:DF} and Rem.~\ref{rem:propDF}.\end{proof}
\begin{cor}
\label{cor:FgenCocomplete}Let $G:\mathcal{I}\ra\Fgen$ be a functor
from a small category $\mathcal{I}$. Then 
\[
\FGD\left(\colim_{i\in\mathcal{I}}\DFG(G_{i})\right)\simeq\colim_{i\in\mathcal{I}}G_{i}.
\]
 Therefore, the category $\Fgen$ is cocomplete.\end{cor}
\begin{proof}
Since $\FGD$ is a left adjoint, it preserves colimits 
\[
\FGD\left(\colim_{i\in\mathcal{I}}\DFG(G_{i})\right)\simeq\colim_{i\in\mathcal{I}}\FGD\left(\DFG(G_{i})\right)=\colim_{i\in\mathcal{I}}\FGD\left(G_{i}\right).
\]
 But $G_{i}\in\Fgen$ is functionally generated, so $\FGD(G_{i})=G_{i}$
from \ref{enu:XisXhatF} of Rem. \ref{rem:propDF}.\end{proof}
\begin{cor}
\label{cor:FGDPreservesSub} Let $X\in\Diff$ and let $S\in\tau_{X}$
be a $D$-open subset. Then $\FGD(S\prec X)=(S\prec\FGD(X))$.\end{cor}
\begin{proof}
By~\cite[Lem.~3.17]{CSW}, $\tau_{(S\prec X)}=\{A\cap S\mid A\in\tau_{X}\}$
since $S$ is open. By \ref{enu:topOfFunctExtension} of Rem. \ref{rem:propDF}
we have $\tau_{\FGD(X)}=\tau_{X}$ and hence $\TD(\FGD(S\prec X))=\TD(S\prec X)=\TD(S\prec\FGD(X))$.
The smoothness of the identity set map $\FGD(S\prec X)\ra(S\prec\FGD(X))$
follows from Thm. \ref{thm:fgtfulInDiffIsRightAdjoint}, and the smoothness
of the inverse set map essentially follows from \ref{enu:ext} of
Rem. \ref{rem:propDF}.
\end{proof}
Since the coequalizer in $\Fgen$ in general is different from the
coequalizer in $\Diff$, the forgetful functor $\xymatrix{\DFG:\Fgen\ar@{^{(}->}[r] & \Diff}
$ has no right adjoint. Here is another interesting example that colimit
in $\Fgen$ is different from the corresponding colimit in $\Diff$:
\begin{example}
\label{exa:pushout}Let $X$ be the pushout of 
\begin{equation}
\xymatrix{\R & \R^{0}\ar[r]^{0}\ar[l]_{0} & \R}
\label{eq:pushoutCounterEx}
\end{equation}
 in $\Diff$. Then we have a commutative diagram 
\[
\xymatrix{\R^{0}\ar[r]^{0}\ar[d]_{0} & \R\ar[d]^{i}\\
\R\ar[r]_{j} & \R^{2}
}
\]
 in $\Diff$ with $i(x)=(x,0)$ and $j(y)=(0,y)$. This induces a
smooth injective map $X\ra\R^{2}$. Write $Y\in\Diff$ for the image
of this map with the sub-diffeology of $\R^{2}$. One can show that 
\begin{enumerate}[leftmargin={*},label=(\roman*),align=left]
\item the induced smooth map $X\ra Y$ is not a diffeomorphism; 
\item the $D$-topology on both $X$ and $Y$ coincide with the sub-topology
of $\R^{2}$; 
\item for any open subset $A$ of $\R^{2}$, $C^{\infty}(A\cap X\prec X,\R)=C^{\infty}(A\cap Y\prec Y,\R)$,
which implies that $X$ is not functionally generated; 
\item $Y$ is Frölicher because $\Fr$ is closed with respect to subobjects,
so $Y\in\Fgen$. 
\end{enumerate}
\end{example}
Hence, by Cor. \ref{cor:FgenCocomplete}, the pushout of \eqref{eq:pushoutCounterEx}
in $\Fgen$ is $\FGD(X)\simeq Y\not{\!\!\simeq\,\,}X$.

Now we show that the embedding $\FGF:\Fr\ra\Fgen$ is not essentially
surjective, that is, there are functionally generated spaces which
are not from Frölicher spaces: 
\begin{example}
\label{eg:nonfgen} Let $Y=(-\infty,0)\cup(0,\infty)$, and let $X$
be the pushout of 
\[
\xymatrix{\R & ~Y\ \ar@{^{(}->}[r]\ar@{_{(}->}[l] & \R}
\]
 in $\Diff$. Then $C^{\infty}(X,\R)\simeq C^{\infty}(\R,\R)$. Since
no element in $C^{\infty}(X,\R)$ can detect the double points at
origin, there is no Frölicher space such that its image under the
embedding $\DF:\Fr\ra\Diff$ is $X$. But since the two colimit maps
$\R\ra X$ are injective and open, $X$ is functionally generated.
In other words, for any $U\in\OR$, $C^{\infty}(X,\R)$ \emph{can
not} detect whether an arbitrary function $U\ra X$ is smooth, but
it \emph{can} detect whether a continuous function $U\ra X$ is smooth.
Moreover, the initial topology on $X$ with respect to $C^{\infty}(X,\R)$
is strictly coarser than the $D$-topology.\end{example}
\begin{thm}
\label{thm:FgenIsCartClosed}The category $\Fgen$ is Cartesian closed.\end{thm}
\begin{proof}
Since $\Diff$ is Cartesian closed and the product in $\Fgen$ is
the same as the product in $\Diff$, it suffices to show that if $X$
is a diffeological space and $Y$ is a functionally generated space,
then the functional diffeology of the space $\Coo(X,Y)$ is functionally
generated.

We split the proof of the claim into three steps.

\textbf{Step 1}: We prove that if $\Coo(\R^{n},Y)$ is functionally
generated for all $n\in\N$, then $\Coo(X,Y)$ is functionally generated.

To prove that $\Coo(X,Y)$ is functionally generated, by \ref{enu:FgeneratesDequiv}
of Rem. \ref{rem:propDF}, for any $d\In{U}\FGD(\Coo(X,Y))$ we need
to show that $d\In{U}\Coo(X,Y)$, i.e., that $d^{\vee}:U\times X\ra Y$
is smooth. This is equivalent to show that for any plot $p:\R^{n}\ra X$,
the composition 
\[
\xymatrix@C=10ex{U\times\R^{n}\ar[r]^{1_{U}\times p} & U\times X\ar[r]^{d^{\vee}} & Y}
\]
 is a plot of $Y$. This is again equivalent to show that the composition
\[
\xymatrix@C=10ex{U\ar[r]\sp(0.38){d} & \Coo(X,Y)\ar[r]^{p^{*}} & \Coo(\R^{n},Y)}
\]
 is smooth. By assumption $\Coo(\R^{n},Y)$ is functionally generated,
and the map $p^{*}$ is smooth, so the adjunction $\FGD\dashv\DFG$
(Thm. \ref{thm:fgtfulInDiffIsRightAdjoint}) implies that $p^{*}:\FGD(\Coo(X,Y))\ra\Coo(\R^{n},Y)$
is smooth. But $d\In{U}\FGD(\Coo(X,Y))$. So $p^{*}\circ d:U\ra\Coo(\R^{n},Y)$
is smooth, which prove our first claim.

\textbf{Step 2}: We prove below that if $d:U\ra\Coo(\R^{n},Y)$ is
a continuous map, then the induced function $d^{\vee}:U\times\R^{n}\ra Y$
is continuous.

Let $A$ be a $D$-open subset of $Y$, and let $(u,x)\in\left(d^{\vee}\right)^{-1}(A)$.
Since $d(u)\in\Coo(\R^{n},Y)$, $(d(u))^{-1}(A)$ is an open neighborhood
of $x\in\R^{n}$. Take a relatively compact open neighborhood $V$
of $x\in\R^{n}$ such that its closure $\bar{V}\subseteq(d(u))^{-1}(A)$.
Write $\tilde{A}=\{f\in\Coo(\R^{n},Y)\mid f(\bar{V})\subseteq A\}$.
Since the $D$-topology on $\Coo(\R^{n},Y)$ contains the compact-open
topology (\cite[Prop.~4.2]{CSW}), $\tilde{A}$ is $D$-open in $\Coo(\R^{n},Y)$.
Hence, $W:=d^{-1}(\tilde{A})$ is an open neighborhood of $u\in U$.
Therefore, $W\times V$ is an open neighborhood of $(u,x)\in\left(d^{\vee}\right)^{-1}(A)$,
which implies that the map $d^{\vee}$ is continuous.

\textbf{Step 3}: We prove below that $\Coo(\R^{n},Y)$ is functionally
generated.

Let $d\In{U}\FGD(\Coo(\R^{n},Y))$.We need to show that the induced
function $d^{\vee}:U\times\R^{n}\ra Y$ is smooth. From Step 2, we
know that $d^{\vee}$ is continuous. Since $Y$ is functionally generated,
it is enough to show that for any $D$-open subset $A$ of $Y$ and
any $l\in\Coo(A\prec Y,\R)$, the composition 
\[
\xymatrix@C=10ex{\left(d^{\vee}\right)^{-1}(A)\ar[r]\sp(0.57){d^{\vee}|_{\left(d^{\vee}\right)^{-1}(A)}} & A\ar[r]^{l} & \R}
\]
 is smooth. For any $(u,x)\in\left(d^{\vee}\right)^{-1}(A)$, use
the notations $\tilde{A}$, $V$ and $W$ defined in Step 2. Since
smoothness is a local condition, it is enough to show that the composition
\[
\xymatrix@C=10ex{W\times V\ \ar@{^{(}->}[r] & \left(d^{\vee}\right)^{-1}(A)\ar[r]\sp(0.57){d^{\vee}|_{\left(d^{\vee}\right)^{-1}(A)}} & A\ar[r]^{l} & \R}
\]
 is smooth. Equivalently, we need to show that the composition 
\[
\xymatrix{W\ar[r]\sp(0.3){d|_{W}} & (\tilde{A}\prec\Coo(\R^{n},Y))\ar[r]\sp(0.58){\text{Res}} & \Coo(V,A)\ar[r]^{l_{*}} & \Coo(V,\R)}
\]
 is smooth, where $\text{Res}$ is the restriction map. It is easy
to see that the map $\text{Res}:(\tilde{A}\prec\Coo(\R^{n},Y))\ra\Coo(V,A)$
is smooth, so $l_{*}\circ\text{Res}:(\tilde{A}\prec\Coo(\R^{n},Y))\ra\Coo(V,\R)$
is smooth. Since $d\In{U}\FGD(\Coo(\R^{n},Y))$, we get $d|_{W}\In{W}(\tilde{A}\prec\FGD(\Coo(\R^{n},Y)))$.
But both $V$ and $\R$ are Frölicher spaces, so $\Coo(V,\R)$ is
functionally generated, and the adjunction $\FGD\dashv\DFG$ (Thm.
\ref{thm:fgtfulInDiffIsRightAdjoint}) implies that the map $l_{*}\circ\text{Res}:\FGD(\tilde{A}\prec\Coo(\R^{n},Y))\ra\Coo(V,\R)$
is smooth. By Cor. \ref{cor:FGDPreservesSub} we have $\FGD(\tilde{A}\prec\Coo(\R^{n},Y))=(\tilde{A}\prec\FGD(\Coo(\R^{n},Y)))$,
so the conclusion follows.
\end{proof}

\subsection{\label{sub:Preservation-of-limits}Preservation of limits and (suitable)
colimits of manifolds}

In this subsection, we are going to discuss the question that if a
limit (or colimit) exists in $\Man$, the category of smooth manifolds
and smooth maps, then is it the same as the corresponding limit (or
colimit) in $\Fgen$? The statements of the main results and the idea
of proofs mainly come from \cite{ncatlab}.
\begin{thm}
\cite{ncatlab}\label{thm:limitsPreserv}Let $F:\mathcal{I}\ra\Man$
be a functor. Assume that $\lim F$ exists in $\Man$. Write $\FGM:\Man\ra\Fgen$
for the embedding functor. Then $\FGM(\lim F)\simeq\lim(\FGM\circ F)$.\end{thm}
\begin{proof}
By the universal property of limit in $\Fgen$, there is a canonical
smooth map $\eta:\FGM(\lim F)\ra\lim(\FGM\circ F)$.

First we prove that $|\eta|$ is surjective. Note that any $x\in|\lim(\FGM\circ F)|$
corresponds to a smooth map $x:\R^{0}\ra\lim(\FGM\circ F)$. So we
have a cone $x\ra F$. Since $\R^{0}$ is a smooth manifold and $\lim F$
exists in $\Man$, by the universal property of limit in $\Man$ and
$\Fgen$, there exists $y:\R^{0}\ra\lim F$ such that $x=\eta\circ\FGM(y)$,
which implies that $|\eta|$ is surjective.

Next we prove that $|\eta|$ is injective. If $a,a'\in|\FGM(\lim F)|$
such that $|\eta|(a)=|\eta|(a')$, then the two cones $a\ra F$ and
$a'\ra F$ have the same image in the target. By the universal property
of limit in $\Man$, $a=a'$.

Finally, we prove that $\eta^{-1}$ is smooth. Let $d\in_{U}\lim(\FGM\circ F)$.
Since the functor $\FGM$ is fully faithful, we get a cone $U\ra F$.
Note that $\FGM(U)=U$. By the universal property of limit in $\Man$
and $\Fgen$, we get a smooth map $f:U\ra\lim F$ such that $\eta^{-1}\circ d=\FGM(f)$.
Hence, $\eta^{-1}$ is smooth.

Therefore, $\FGM(\lim F)\simeq\lim(\FGM\circ F)$.\end{proof}
\begin{rem}
Note that the category $\OR$ with the usual open coverings is a site.
\cite{Ba-Ho11} showed that $\OR$ is a concrete site, and the category
$\Diff$ is equivalent to the category $\CSh(\OR)$ of concrete sheaves
over $\OR$.

We write $\CPre(\OR)$, $\Pre(\OR)$ and $\Sh(\OR)$ for the category
of concrete presheaves over $\OR$, the category of presheaves over
$\OR$ and the category of sheaves over $\OR$, respectively. There
are embedding functors 
\[
\small{\xymatrix{\Man\ar[r] & \Fr\ar[r] & \Fgen\ar[r] & \Diff\ar[r] & \CPre(\OR)\ar[r] & \Pre(\OR)}
}
\]
 and 
\[
\small{\xymatrix{\Man\ar[r] & \Diff\ar[r] & \Sh(\OR).}
}
\]
 Moreover, if a limit exists in $\Man$, then the corresponding limits
in all the other categories listed above are isomorphic to that limit
in the corresponding categories.
\end{rem}
In general, if a colimit in $\Man$ exists, when viewed as a functionally
generated space it may be different from the corresponding colimit
in $\Fgen$.
\begin{example}
Recall that in Ex.~\ref{eg:nonfgen}, we showed that the underlying
set of the pushout $X$ of 
\[
\xymatrix{\R & ~\R\setminus\{0\}\ \ar@{^{(}->}[r]\ar@{_{(}->}[l] & \R}
\]
 in $\Fgen$ has double points at origin. Moreover, the $D$-topology
$\tau_{X}$ is not Hausdorff. One can also show by continuity that
the pushout of this diagram in $\Man$ exists, and it is $\R$. Therefore,
$X\not{\!\!\simeq\,\,}\R$ in $\Fgen$.\end{example}
\begin{thm}
\cite{ncatlab}\label{thm:surjective} Let $G:\mathcal{J}\ra\Man$
be a functor such that $\colim G$ exists in $\Man$. Then the canonical
smooth map 
\[
\eta:\colim(\FGM\circ G)\ra\FGM(\colim G)
\]
 induces a surjective map 
\[
|\eta|:|\colim(\FGM\circ G)|\ra|\FGM(\colim G)|.
\]
\end{thm}
\begin{proof}
The canonical smooth map $\eta:\colim(\FGM\circ G)\ra\FGM(\colim G)$
comes from the universal property of colimits in $\Man$ and $\Fgen$.

Assume that $|\eta|$ is not surjective. Say $y\in|\FGM(\colim G)|$
is not in the image. Then $A:=\colim G\setminus\{y\}$ is a smooth
manifold, and $G(j)\ra\colim G$ factors through $A\hookrightarrow\colim G$
for each $j\in J$ since $|\colim(\FGM\circ G)|=\colim|\FGM\circ G|$.
Hence, by the universal property of colimit in $\Man$, the identity
map $\colim G\ra\colim G$ must factor through $A\hookrightarrow\colim G$,
which is impossible. Therefore, $|\eta|$ is surjective.\end{proof}
\begin{thm}
\cite{ncatlab}\label{thm:preservecolimits} Let $G:\mathcal{J}\ra\Man$
be a functor such that $\colim G$ exists in $\Man$. If $\colim(\FGM\circ G)$
is a Frölicher space and its $D$-topology is Hausdorff, then $\colim(\FGM\circ G)\simeq\FGM(\colim G)$. \end{thm}
\begin{proof}
Let $X$ be a diffeological space. It is direct to show that
\begin{enumerate}
\item If $\Coo(X,\R)$ separates points, that is, 
\[
\forall x\neq x'\in X,\exists l\in\Coo(X,\R):l(x)\neq l(x'),
\]
 then the $D$-topology $\tau_{X}$ is Hausdorff.

\item If the $D$-topology $\tau_{X}$ is Hausdorff, then any plot $\R\ra X$
with finite image must be constant.

\item If any plot $\R\ra X$ with finite image must be constant and $X$
is Frölicher, then $\Coo(X,\R)$ separates points.
\end{enumerate} 

By Thm.~\ref{thm:surjective}, the canonical smooth map 
\[
\eta:\colim(\FGM\circ G)\ra\FGM(\colim G)
\]
 induces a surjective map 
\[
|\eta|:|\colim(\FGM\circ G)|\ra|\FGM(\colim G)|.
\]

For any $l\in\Coo(\colim(\FGM\circ G),\R)$, by the universal property
of colimits in $\Man$ and $\Fgen$, there exists a unique smooth
map $f:\colim G\ra\R$ such that $\FGM(f)\circ\eta=l$.

Since $\colim(\FGM\circ G)$ is Frölicher and its $D$-topology is
Hausdorff, from the above we know that $\Coo(\colim(\FGM\circ G),\R)$
separates points. Then the equality $\FGM(f)\circ\eta=l$ implies
that $|\eta|$ is injective.

For any $d\in_{U}\FGM(\colim G)$, $l\circ\eta^{-1}\circ d=\FGM(f)\circ d$
is smooth for any $l\in\Coo(\colim(\FGM\circ G),\R)$. Since $\colim(\FGM\circ G)$
is Frölicher, $\eta^{-1}\circ d\in_{U}\colim(\FGM\circ G)$. Hence,
$\eta^{-1}$ is smooth.

Therefore, $\colim(\FGM\circ G)\simeq\FGM(\colim G)$.
\end{proof}
Here is an immediate application:
\begin{example}
Recall from Ex.~\ref{exa:pushout} that the pushout of 
\[
\xymatrix{\R & \R^{0}\ar[l]_{0}\ar[r]^{0} & \R}
\]
 in $\Fgen$ is the union of the two axes in $\R^{2}$ with the sub-diffeology,
which is a Frölicher space with Hausdorff $D$-topology, but clearly
not a smooth manifold. By Thm.~\ref{thm:preservecolimits}, the pushout
of the above diagram does not exist in $\Man$.
\end{example}

\subsection{\label{sub:Categorical-frameworks}Categorical frameworks for generalized
functions}

We start by asking the following question: What is a ``good'' category
to frame spaces like $\D(\Omega)$, $\D'(\Omega)$, $\mathcal{A}_{q}(\Omega)$,
$U(\Omega)$, $\gs(\Omega)$ and $\gf(\Omega)$?

The following list of remarks permits to restrict the range of choices:
\begin{rem}
\label{rem:catFrameworksGF}\ 
\begin{enumerate}[leftmargin={*},label=(\roman*),align=left]
\item Schwartz distribution theory is classically framed using locally
convex topological vector spaces (LCTVS), so it is natural to search
for a category which contains the category $\LCS$ of LCTVS and continuous
linear maps as a subcategory. 
\item \label{enu:A_0Affine}The space $\mathcal{A}_{0}(\Omega)$ is an affine
space and is usually identified with its underlying vector space (see
e.g.\ \cite{GKOS}). However, it seems that the necessity of this
identification is only due to the choice of a category like $\LCS$,
which is not closed with respect to arbitrary subspaces. It would
be better to choose a complete category. 
\item $U(\Omega)$ and $\mathcal{A}_{0}(\Omega)$ can be viewed as manifolds
modelled in convenient vector spaces (CVS, \cite{FK,KM}). However,
the category of this type of manifolds is not Cartesian closed (\cite{KM}),
whereas Cartesian closedness is a basic choice preferred by several
mathematicians working with infinite dimensional spaces (see e.g.\ \cite{Gio10c}
and references therein). 
\item The candidate category shall contain the category $\ConInf$ of convenient
vector spaces and generic smooth maps (\cite{FK,KM}) as a (full)
subcategory because the differential calculus of these spaces is used
in the study of Colombeau algebras (\cite{GKOS}). Note that we have
embeddings $\Con\subseteq\LCS\subseteq\ConInf$, where $\Con$ is
the category of CVS and continuous linear maps. 
\item The candidate category must be closed with respect to arbitrary quotient
spaces, so as to contain the quotient algebras $\gs(\Omega)$ and
$\gf(\Omega)$. Of course, a better choice would be to consider a
cocomplete category. Since, generally speaking, CVS are not closed
with respect to quotient spaces (see \cite[page 22]{KM}), the candidate
category cannot be $\ConInf$. 
\item The candidate category must also contain nonlinear maps like the product
of GF, e.g. 
\[
(u,v)\in\gs(\Omega)\times\gs(\Omega)\mapsto u\cdot v\in\gs(\Omega)
\]
 or, more generally, any nonlinear smooth operation $f\in\Coo(\R^{n},\R)$
which can be extended to an operation of our algebras of GF, e.g.
\[
(u_{1},\ldots,u_{n})\in(\gs(\Omega))^{n}\mapsto\left[f(u_{1\eps}(-),\ldots,u_{n\eps}(-))\right]\in\gs(\Omega).
\]
 Another feature of the candidate category we are looking for is to
contain as arrows the maps between infinite dimensional spaces like
convolutions, derivatives and integrals of GF.
\end{enumerate}
\end{rem}
In the literature, there are only two categories satisfying all these
requirements: the category $\Diff$ of diffeological spaces, and the
category $\Fr$ of Frölicher spaces. In the present work, we will
also introduce the category $\Fgen$ of functionally generated spaces
as another framework for GF, trying to take the best ideas and properties
of both $\Diff$ and $\Fr$. We will see that all the spaces $\D_{K}(\Omega)$,
$\D(\Omega)$, $\D'(\Omega)$, $\mathcal{A}_{q}(\Omega)$, $U(\Omega)$,
$\gs(\Omega)$ and $\gf(\Omega)$ are objects of these categories,
and in this paper, we in particular study them as functionally generated
spaces.

\section{Topologies for spaces of generalized functions}

\cite[page 2]{KM} declared that ``locally convex topology is not
appropriate for non-linear questions in infinite dimensions''. Indirectly,
this is also confirmed by the fact that topology plays a less important
role in categories like $\Diff$ or $\Fr$. The main aim of this section
is to highlight some relationship between Cartesian closedness and
locally convex topology.

\subsection{Locally convex vector spaces and Cartesian closed categories}

The problems that arise in relating locally convex topology and Cartesian
closedness can be expressed as follows:
\begin{thm}
\label{thm:CartClosAndLCTop}Let $F\in\LCS$, and let $(\mathcal{T},U)$
be a Cartesian closed concrete category over $\Top$, with exponential
objects given by the hom-functor $\mathcal{T}(-,-)$ and the forgetful
functor $U:\mathcal{T}\ra\Top$ acting as identity on arrows. Assume
that $R$, $\bar{F}\in\mathcal{T}$, $U(R)=\R$ and $U(\bar{F})=F$.
Set $F':=\LCS(F,\R)$ for the continuous dual of $F$ and assume that
\begin{equation}
F'\subseteq\left|U(\mathcal{T}(\bar{F},R))\right|,\label{eq:FPrimeInTau}
\end{equation}
 
\[
\left|U(\bar{F}\times\mathcal{T}(\bar{F},R))\right|=\left|U(\bar{F})\times U(\mathcal{T}(\bar{F},R))\right|
\]
 and the topology of the space $U(\bar{F}\times\mathcal{T}(\bar{F},R))$
is coarser than the product topology of $U(\bar{F})\times U(\mathcal{T}(\bar{F},R))$.
Finally assume that for all $g\in F'$ the map $(\lambda\in\R\mapsto\lambda\cdot g\in F')$
is continuous with respect to the topology induced on $F'$ by \eqref{eq:FPrimeInTau}.
Then the locally convex topology on the space $F$ is normable.\end{thm}
\begin{proof}
The idea for the proof is only a reformulation of the corresponding
result in \cite[page~2]{KM}\foreignlanguage{english}{. Since $\mathbf{\mathcal{T}}$
is Cartesian closed, every evaluation 
\[
\text{ev}_{XY}(x,f):=f(x)\quad\forall x\in X\:\forall f\in\mathbf{\mathcal{T}}(X,Y)
\]
 is an arrow of $\mathbf{\mathcal{T}}$ (this is a general result
in every Cartesian closed category, see e.g. \cite{AHS}). Thus, $U(\text{ev}_{XY})=\text{ev}_{XY}$
is a continuous function. In particular, $\text{ev}_{\bar{F}R}:U(\bar{F}\times\mathcal{T}(\bar{F},R))\ra U(R)=\R$
is continuous. By assumption, also $\text{ev}_{\bar{F}R}:F\times U(\mathcal{T}(\bar{F},R))\ra\R$
is continuous. Therefore, also its restriction to the subspace $F'=\LCS(F,\R)\subseteq\left|U(\mathbf{\mathcal{T}}(\bar{F},R))\right|$
is (jointly) continuous: 
\[
\varepsilon:=\text{ev}_{\bar{F}R}|_{F\times F'}:F\times F'\ra\R.
\]
 Hence, we can find neighborhoods $U\subseteq F$ and $V\subseteq F'$
of zero such that $\varepsilon(U\times V)\subseteq[-1,1]$, that is
\[
U\subseteq\left\{ u\in F\mid\forall f\in V:\,\,|f(u)|\le1\right\} .
\]
 But then, because the map $(\lambda\in\R\mapsto\lambda\cdot g\in F')$
is continuous, taking a generic functional $g\in F'$, we can always
find $\lambda\in\R_{\ne0}$ such that $\lambda g\in V$, and hence
$|g(u)|\le1/\lambda$ for every $u\in U$. Any continuous functional
is thus bounded on $U$, so the neighborhood $U$ itself is bounded
(see e.g. \cite{KN}). Since the topology of any locally convex vector
space with a bounded neighborhood of zero is normable (see e.g. \cite{KN}),
we get the conclusion.}
\end{proof}

If, in this theorem, we take $F=\Coo(\R,\R)$ or $F=\D(\Omega)$ or
any other non-normable LCTVS, there are two possibilities to make
the space $F$ an object in a Cartesian closed category: 

\begin{enumerate}[leftmargin={*},label=(\alph*),align=left]
\item $F$ belongs to a Cartesian closed category $\mathcal{T}$, but $\mathcal{T}$
is not a concrete category over $\Top$. This is the solution used
in CVS theory which are embedded in the Cartesian closed category
$\ConInf$. Note that $\LCS$ is not a full subcategory of $\ConInf$
since not every arrow of $\ConInf$ is continuous. A typical example
of a $\ConInf$-smooth but not continuous map (with respect to the
given locally convex topology instead of the $D$-topology) is the
evaluation 
\[
\text{ev}:(x,g)\in F\times F'\mapsto g(x)\in\R.
\]

\item For the solution adapted by using diffeological spaces, we can take
$\mathcal{T}=\Diff$. But then several assumptions of Thm.~\ref{thm:CartClosAndLCTop}
fail: e.g.\ in general $\tau_{X\times Y}\supseteq\tau_{X}\times\tau_{Y}$,
but not the opposite as required; moreover, the $D$-topology on $\D(\Omega)$
is not normable since it is finer than the usual locally convex topology,
which is also not normable. On the contrary, there is no problem regarding
the continuity of the product by scalar, as stated in the following:\end{enumerate}
\begin{thm}
\label{thm:DtopAndTVS}Let $F$ be any one of the spaces $\Coo(\Omega,\R)$
or $\D(\Omega)$, and let $\tau_{F}$ be the $D$-topology on $F$.
Let 
\[
F'_{\text{\emph{s}}}:=\left(\left\{ l\in\Coo(F,\R)\mid l\text{ is linear}\right\} \prec\Coo(F,\R)\right)
\]
 be the smooth dual of $F$, and let $\tau_{F'_{\text{\emph{s}}}}$
be the $D$-topology on $F'_{\text{\emph{s}}}$. Then, with respect
to pointwise operations, both spaces $(F,\tau_{F})$ and $(F'_{\text{\emph{s}}},\tau_{F'_{\text{\emph{s}}}})$
are topological vector spaces.\end{thm}
\begin{proof}
We proceed for the case $F=\Coo(\Omega,\R)$ since the other one is
very similar. For simplicity set $Y^{X}:=\Coo(X,Y)$, and 
\[
\langle-,-\rangle:(u,v)\in F\times F\mapsto\left(r\in\Omega\mapsto\left(u(r),v(r)\right)\in\R^{2}\right)\in\left(\R^{2}\right)^{\Omega}
\]
 
\[
\gamma_{1}:(u,v)\in\left(\R^{2}\right)^{\Omega}\times\R^{\R^{2}}\mapsto v\circ u\in F
\]
 
\[
\gamma_{2}:(u,v)\in\R^{\R}\times F\mapsto u\circ v\in F
\]
 
\[
s_{\R}:(r,s)\in\R^{2}\mapsto r+s\in\R
\]
 
\[
p_{\R}:(r,s)\in\R^{2}\mapsto r\cdot s\in\R
\]
 
\begin{multline*}
(-,-):(u,v)\in F^{\R\times F}\times F^{\R\times F}\mapsto\\
\left((\lambda,f)\in\R\times F\mapsto\left(u(\lambda,f),v(\lambda,f)\right)\in F\times F\right)\in\left(F\times F\right)^{\R\times F}.
\end{multline*}
 It is easy to prove that the pointwise sum and pointwise product
by scalars are given by $(-)+(-)=\gamma_{1}(-,s_{\R})\circ\langle-,-\rangle$
and $(-)\cdot(-)=\gamma_{2}\circ\left(p_{\R}^{\wedge}\circ q_{1},q_{2}\right)$,
where $q_{1}:\R\times F\ra\R$ and $q_{2}:\R\times F\ra F$ are the
projections. Therefore, both sum and product in $F$ are composition
or pairing of smooth functions, and hence they are smooth and continuous
in the $D$-topology. Analogously, we can proceed with the smooth
dual $F'_{\text{s}}$ by considering the properties of the operator
$(-\prec-)$.
\end{proof}

\section{Spaces of compactly supported functions as functionally generated
spaces}

It is very easy to see that the spaces $\D(\Omega)=\{f\in\Coo(\Omega,\R)\mid\supp(f)\Subset\Omega\}$
and $\D_{K}(\Omega)=\{f\in\D(\Omega)\mid\supp(f)\Subset K\}$ with
$K\Subset\Omega$ are functionally generated spaces.
Recall that $\D_{K}(\Omega)$ is a LCTVS whose topology is induced
by the family of norms (Fréchet structure) 
\begin{equation}
\|\phi\|_{K,m}:=\max_{|\alpha|\le m}\max_{x\in K}\|\partial^{\alpha}\phi(x)\|\qquad\forall\phi\in\D_{K}(\Omega)\ \forall m\in\N.\label{eq:FrechetOnD_K}
\end{equation}
 Also the space $\D(\Omega)$ is a LCTVS obtained as the inductive
limit (colimit) of $\D_{K}(\Omega)$ for $K\Subset\Omega$, i.e.:
\begin{equation}
A\text{ is open in }\D(\Omega)\iff\forall K\Subset\Omega:\ i_{K}^{-1}(A)=A\cap\D_{K}(\Omega)\text{ is open in }\D_{K}(\Omega),\label{eq:LCTopOnD}
\end{equation}
 where $i_{K}:\D_{K}(\Omega)\hookrightarrow\D(\Omega)$. Recall that
on the space $\D'(\Omega)$ of distributions (i.e., linear maps $l:\left|\D(\Omega)\right|\ra\R$
which are continuous with respect to the locally convex topology)
there is a topology called \emph{weak$^{*}$ topology}, i.e., the
coarsest topology such that each evaluation $\text{ev}_{\phi}:u\in\D'(\Omega)\mapsto\langle u,\phi\rangle\in\R$
is continuous. With respect to this topology, $\D'(\Omega)$ is a
LCTVS.

The so-called \emph{canonical diffeology} on these spaces is a particular
case of the following:
\begin{defn}
\label{def:canonicalDiffeolTVS}Let $V$ be a topological vector space.
The \emph{canonical diffeology} $\mathbb{D}(V)=\bigcup_{U\in\OR}\mathbb{D}_{U}(V)$,
is given by the sets $\mathbb{D}_{U}(V)$ of all maps $d:U\ra V$
which are smooth when tested by continuous linear functionals, i.e.,
\[
\forall l:V\ra\R\text{ continuous linear: }l\circ d\in\Coo(U,\R).
\]
 Therefore (see Def. \ref{def:DF} and Rem. \ref{rem:propDF}), $(V,\mathbb{D}(V))\in\Fgen$
and, since the functionals are globally defined, this is also a Frölicher
space, i.e., 
\[
\DF(\FD(V,\mathbb{D}(V)))=(V,\mathbb{D}(V)).
\]

\end{defn}
We will continue to denote our spaces by $\D(\Omega)$ and $\D_{K}(\Omega)$
even when we think of them as diffeological spaces with the canonical
diffeology. When we want to underscore that we are considering them
only as LCTVS with the topology given by \eqref{eq:LCTopOnD} and
\eqref{eq:FrechetOnD_K}, we will use the notations $\Dlc(\Omega)$
and $\Dlc_{K}(\Omega)$.

\subsection{Plots of $\D_{K}(\Omega)$, $\D(\Omega)$ and Cartesian closedness}

It is also interesting to reformulate the property of being a plot
$d\In{U}\D(\Omega)$, or $d\In{U}\D_{K}(\Omega)$, using Cartesian
closedness. This permits to compare better the canonical diffeology
on these spaces as LCTVS with the diffeology induced on them as subspaces
of $\Coo(\Omega,\R)$. We will denote with $\Dsmooth(\Omega)$ this
diffeological space, so that $d\In{U}\Dsmooth(\Omega)$ if and only
if $d:U\ra|\D(\Omega)|$ and $i\circ d\In{U}\Coo(\Omega,\R)$, where
$i:\D(\Omega)\hookrightarrow\Coo(\Omega,\R)$ is the inclusion. Recall
that $i\circ d\In{U}\Coo(\Omega,\R)$ holds if and only if $\left(i\circ d\right)^{\vee}\in\Coo(U\times\Omega,\R)$.
Analogously, we define $\Dsmooth_{K}(\Omega)$. In performing this
comparison, we will use Lem.~2.1, Lem.~2.2 and Thm.~2.3 of \cite{KR06}
which are cited here for reader's convenience. In this subsection,
without confusion we use the same notation for morphisms in different
categories when the functions for the underlying sets are the same.
\begin{lem}[2.1 of \cite{KR06}]
\label{lem:2.1KR}If $U\in\OR$ and $f\in\Coo(U,\D(\Omega))$, then
$f:U\ra\Dlc(\Omega)$ is continuous.
\end{lem}
To state the other cited results of \cite{KR06}, we need the following:
\begin{defn}
\label{def:UnifBoundedSupp}Let $U\in\OR$ and let $f:U\times\Omega\ra\R$
be a map. We say that 
\[
f\text{ is \emph{of uniformly bounded support} (with respect to }U\text{)}
\]
 if 
\[
\exists K\Subset\Omega\,\forall u\in U:\ \supp(f(u,-))\subseteq K.
\]
 We say that 
\[
f\text{ is \emph{locally of uniformly bounded support}}
\]
 if 
\[
\forall u\in U\,\exists V\text{open neigh. of }u\text{ in }U:\ f|_{V\times\Omega}\text{ is of uniformly bounded support}.
\]
 Finally we say that 
\[
f\text{ is \emph{pointwise of bounded support}}
\]
 if 
\[
\forall u\in U\,\exists K\Subset\Omega:\ \supp(f(u,-))\subseteq K.
\]

\end{defn}
Using this definition, we can state
\begin{lem}[2.2 of \cite{KR06}]
\label{lem:2.2KR}Let $U\in\OR$ and assume that $f\in\Coo(U\times\Omega,\R)$
is pointwise of bounded support. Then the following are equivalent 
\begin{enumerate}[leftmargin={*},label=(\roman*),align=left]
\item $f$ is locally of uniformly bounded support; 
\item $f^{\wedge}:U\ra\Dlc(\Omega)$ is continuous. 
\end{enumerate}
\end{lem}
\begin{thm}[2.3 of \cite{KR06}]
\label{thm:2.3KR}Let $U\in\OR$. Then the following are equivalent: 
\begin{enumerate}[leftmargin={*},label=(\roman*),align=left]
\item $f\in\Coo(U,\D(\Omega))$; 
\item $f^{\vee}\in\Coo(U\times\Omega,\R)$ and $f^{\vee}$ is locally of
uniformly bounded support.
\end{enumerate}
\end{thm}
In other words, Thm~\ref{thm:2.3KR} says that $d\In{U}\D(\Omega)$
if and only if $d\In{U}\Dsmooth(\Omega)$ and $d^{\vee}$ is locally
of uniformly bounded support, and hence we have $\D(\Omega)\subseteq\Dsmooth(\Omega)$.

From these results we can also solve the same problem for the spaces
$\D_{K}(\Omega)$ and $\Dsmooth_{K}(\Omega)$. The following lemma
is analogous to Lem.~\ref{lem:2.1KR} for $\D_{K}(\Omega)$.
\begin{lem}
\label{lem:AnalogOf2.1KRForD_K} If $f\in\Coo(U,\D_{K}(\Omega))$
with $U\in\OR$ and $K\Subset\Omega$, then $f:U\ra\Dlc_{K}(\Omega)$
is continuous.\end{lem}
\begin{proof}
Since the inclusion map $i_{K}:\Dlc_{K}(\Omega)\hookrightarrow\Dlc(\Omega)$
is continuous linear, by post-composition it also takes continuous
linear functionals $l:\Dlc(\Omega)\ra\R$ into continuous linear functionals
$l\circ i_{K}:\Dlc_{K}(\Omega)\ra\R$. From Thm.~\ref{thm:codomFgenTestFunctionals}
it follows that $i_{K}\in\Coo(\D_{K}(\Omega),\D(\Omega))$ and hence
$i_{K}\circ f\in\Coo(U,\D(\Omega))$. Therefore, Lem.~\ref{lem:2.1KR}
implies that $i_{K}\circ f$ is continuous and hence the conclusion
since the topology on $\Dlc_{K}(\Omega)$ coincides with the initial
topology induced by $i_{K}$.
\end{proof}
The following lemma is analogous to Lem.~\ref{lem:2.2KR} for $\D_{K}(\Omega)$.
\begin{lem}
\label{lem:AnalogOf2.2KRForD_K}Let $U\in\OR$ and let $f\in\Coo(U\times\Omega,\R)$.
If there exists $K\Subset\Omega$ such that 
\begin{equation}
\forall u\in U:\ \supp(f(u,-))\subseteq K\label{eq:AnalogOf2.2KR_suppCond}
\end{equation}
 then $f^{\wedge}:U\ra\Dlc_{K}(\Omega)$ is continuous.\end{lem}
\begin{proof}
Clearly $f$ is also locally of uniformly bounded support. Apply Lem.~\ref{lem:2.2KR},
we know that $i_{K}\circ f^{\wedge}:U\ra\Dlc(\Omega)$ is continuous.
Since $\Dlc_{K}(\Omega)$ has the initial topology from $i_{K}:\Dlc_{K}(\Omega)\hookrightarrow\Dlc(\Omega)$,
$f^{\wedge}:U\ra\Dlc_{K}(\Omega)$ is continuous.
\end{proof}
Finally, the following theorem is analogous to Thm.~\ref{thm:2.3KR}
for $\D_{K}(\Omega)$.
\begin{thm}
\label{thm:AnalogOf2.3KRForD_K}Let $U\in\OR$ and let $K\Subset\Omega$.
Then the following are equivalent: 
\begin{enumerate}[leftmargin={*},label=(\roman*),align=left]
\item \label{enu:AnalogOf2.3KR_DiffSmooth}$f\in\Coo(U,\D_{K}(\Omega))$; 
\item \label{enu:AnalogOf2.3KR_OrdSmooth}$f^{\vee}\in\Coo(U\times\Omega,\R)$
and $\supp(f^{\vee}(u,-))\subseteq K$ for all $u\in U$. 
\end{enumerate}
\end{thm}
\begin{proof}
\ref{enu:AnalogOf2.3KR_DiffSmooth} $\Rightarrow$ \ref{enu:AnalogOf2.3KR_OrdSmooth}.
We already proved in Lem.~\ref{lem:AnalogOf2.1KRForD_K} that the
inclusion map $i_{K}\in\Coo(\D_{K}(\Omega),\D(\Omega))$, so $i_{K}\circ f\in\Coo(U,\D(\Omega))$.
By Thm.~\ref{thm:2.3KR} we have $\left(i_{K}\circ f\right)^{\vee}=f^{\vee}\in\Coo(U\times\Omega,\R)$.
The second part of the conclusion follows from the codomain $\D_{K}(\Omega)$
of $f$ in \ref{enu:AnalogOf2.3KR_DiffSmooth}.\\
 \ref{enu:AnalogOf2.3KR_OrdSmooth} $\Rightarrow$ \ref{enu:AnalogOf2.3KR_DiffSmooth}.
Assumption \ref{enu:AnalogOf2.3KR_OrdSmooth} implies that $f$ is
locally of uniformly bounded support. From Thm.~\ref{thm:2.3KR}
we thus obtain that $f\in\Coo(U,\D(\Omega))$. But our assumption
implies that $f(U)\subseteq\left|\D_{K}(\Omega)\right|$. So the conclusion
follows from the following Lem.~\ref{lem:D_KIsD_KPrecD}.
\begin{lem}
\label{lem:D_KIsD_KPrecD}If $K\Subset\Omega$, then $\left(\left|\D_{K}(\Omega)\right|\prec\D(\Omega)\right)=\D_{K}(\Omega)$.\end{lem}
\begin{proof}
We have to prove that figures of both spaces are equal.\\
$\left(\left|\D_{K}(\Omega)\right|\prec\D(\Omega)\right)\supseteq\D_{K}(\Omega)$:
This follows directly from the fact that the inclusion map $i_{K}\in\Coo(\D_{K}(\Omega),\D(\Omega))$.\\
$\left(\left|\D_{K}(\Omega)\right|\prec\D(\Omega)\right)\subseteq\D_{K}(\Omega)$:
Assume that $d:U\ra\D_{K}(\Omega)$ is a map such that $i_{K}\circ d\In{U}\D(\Omega)$,
i.e., $\lambda\circ i_{K}\circ d\in\Coo(U,\R)$ for all $\lambda\in\D'(\Omega)$.
We need to prove that $l\circ d\in\Coo(U,\R)$ for all continuous
linear maps $l:\Dlc_{K}(\Omega)\ra\R$. So the problem is to extend
any such given $l$ to some $\lambda\in\D'(\Omega)$. To this end,
we can repeat the usual proof of the local form of distributions as
derivatives of continuous functions to obtain the following:\end{proof}
\begin{thm}
\label{thm:localFormDPrime_K}For any continuous linear map $l:\Dlc_{K}(\Omega)\ra\R$
there exist $g\in\mC^{0}(\Omega,\R)$ and $\alpha\in\N^{n}$ such
that 
\[
l(\phi)=\langle\partial^{\alpha}g,\phi\rangle\qquad\forall\phi\in\D_{K}(\Omega).
\]
 Therefore, the continuous functional $\langle\partial^{\alpha}g,-\rangle:\Dlc(\Omega)\ra\R$
extends the functional $l$.
\end{thm}

The conclusion follows by applying this theorem.
\end{proof}
\begin{cor}
\label{cor:D_KEqD_Ks}If $K\Subset\Omega$, then $\D_{K}(\Omega)=\Dsmooth_{K}(\Omega)$.\end{cor}
\begin{fact*}
Indeed Thm.~\ref{thm:AnalogOf2.3KRForD_K} says that $d\In{U}\D_{K}(\Omega)$
if and only if $d\In{U}\Dsmooth_{K}(\Omega)$.
\end{fact*}

\subsection{The locally convex topology and the $D$-topology on $\D_{K}(\Omega)$
and $\D(\Omega)$}

In this section we present some results about functionals on the spaces
$\D_{K}(\Omega)$ and $\D(\Omega)$ which are continuous with respect
to the locally convex topology and the $D$-topology. The first result
follows at once from Lem.~\ref{lem:2.1KR} and Lem.~\ref{lem:AnalogOf2.1KRForD_K}:
\begin{cor}
\label{cor:D-TopFinerLCTop}On the spaces $\D_{K}(\Omega)$ and $\D(\Omega)$,
the $D$-topology is finer than the locally convex topology.
\end{cor}
It remains an open problem whether the $D$-topology is strictly finer
than the locally convex topology or not. We first study the behaviour
of maps of the form $\lambda:\D(\Omega)\ra\D(\Omega')$, where henceforth
we always assume that $\Omega'\subseteq\R^{d}$ is open.
\begin{thm}
\label{thm:DIsCVS}\ 
\begin{enumerate}[leftmargin={*},label=(\roman*),align=left]
\item $\D(\Omega)$ is a CVS. 
\item If $T\in\Coo(\D(\Omega),\R)$ is linear, then $T:\Dlc(\Omega)\ra\R$
is continuous. 
\end{enumerate}
\end{thm}
The same results hold for $\D_{K}(\Omega)$.
\begin{proof}
See \cite[page 5, 6, 9]{KR06} or \cite[Lem.~6.2, page 67]{KM}.
\end{proof}
The following lemma is a trivial consequence of Thm.~\ref{thm:codomFgenTestFunctionals},
but we prefer to state it for completeness.
\begin{lem}
\label{lem:linearContImpliesDiffsmooth}If $\lambda:\Dlc(\Omega)\ra\Dlc(\Omega')$
is continuous linear, then $\lambda\in\Coo(\D(\Omega),\D(\Omega'))$.
\end{lem}
In the following results, we show that if a linear map $|\D(\Omega)|\ra\R$
is continuous with respect to the $D$-topology, then it is a distribution:
\begin{thm}
\label{thm:DContImpliesDistr}If $l:T_{D}(\D(\Omega))\ra\R$ is continuous
linear, then 
\[
l\in\Coo(\D(\Omega),\R)\cap\D'(\Omega).
\]

\end{thm}
The schema to prove this theorem is the following: we need to prove
that $l\circ d\in\Coo(U,\R)$ whenever $d\In{U}\D(\Omega)$, i.e.,
by Thm.~\ref{thm:2.3KR}, if $d^{\vee}\in\Coo(U\times\Omega,\R)$
and $d^{\vee}$ is locally of uniformly bounded support. We are going
to prove that: 
\begin{enumerate}[leftmargin={*},label=(\roman*),align=left]
\item \label{enu:limitIdea} For any $u\in U$ the limit 
\begin{equation}
\lim_{h\to0}\frac{d(u+he_{i})-d(u)}{h}\label{eq:limIdeaProof}
\end{equation}
 exists in $T_{D}(\D(\Omega))$, where 
\[
e_{i}=(0,\ptind^{i-1},0,1,0,\ldots,0)\in\R^{n}\supseteq U.
\]
 In fact, this limit is $\left(\frac{\partial d^{\vee}}{\partial e_{i}}\right)^{\wedge}$,
which is again a figure of type $U$ of $\D(\Omega)$.
\item Since $l:T_{D}(\D(\Omega))\ra\R$ is continuous linear, we can apply
\ref{enu:limitIdea} and commute $l$ with the limit and the incremental
ratio to prove that $\frac{\partial}{\partial e_{i}}(l\circ d)$ exists
and is of the form $l\circ p$ with $p\in_{U}\D(\Omega)$. The conclusion
then follows by induction. 
\end{enumerate}
Before proving \ref{enu:limitIdea}, it is indispensable to have the
following:
\begin{lem}
\label{lem:expSpacesHaus}Let $V$ be an open set in $\R^{n}$. Then
the spaces $T_{D}(\Coo(V,\R))$ and $T_{D}(\D(V))$ are Hausdorff.\end{lem}
\begin{proof}
Note that for any $v\in V$, the evaluation maps $l_{v}:h\in\Coo(V,\R)\mapsto h(v)\in\R$
and $\bar{l}_{v}:h\in\D(V)\mapsto h(v)\in\R$ are smooth, and hence
the maps $T_{D}(l_{v}):T_{D}(\Coo(V,\R))\ra\R$ and $T_{D}(\bar{l}_{v}):T_{D}(\D(V))\ra\R$
are both continuous. Therefore, the functional topology on $\Coo(V,\R)$
and $\D(V)$ are Hausdorff. The conclusion then follows by Thm.~\ref{thm:functionaltopology}.
\end{proof}
We now prove Thm.~\ref{thm:DContImpliesDistr}:
\begin{proof}
To prove the existence of the limit in \ref{enu:limitIdea}, we first
fix $d\In{U}\D(\Omega)$, $u\in U$, $e_{i}=(0,\ptind^{i-1},0,1,0,\ldots,0)\in\R^{n}\supseteq U$
and $r\in\R_{>0}$ such that $B_{r}(u)\subseteq U$. Then there exist
an open neighbourhood $V$ of $u$ in $U$ and $a\in\R_{>0}$ such
that $v+he_{i}\in B_{r}(u)$ for all $v\in V$ and $h\in(-a,a)$.
Set $H:=(-a,a)$, and for any $h\in H$, define 
\[
\delta(h):=\left((v,x)\in V\times\Omega\mapsto\int_{0}^{1}\frac{\partial d^{\vee}}{\partial e_{i}}(v+she_{i},x)\diff{s}\in\R\right).
\]
 Clearly $\delta(0)=\frac{\partial d^{\vee}}{\partial e_{i}}|_{V\times\Omega}$.
Thm.~\ref{thm:2.3KR} implies $d^{\vee}\in\Coo(U\times\Omega,\R)$,
so $\delta^{\vee}\in\Coo(H\times V\times\Omega,\R)$ and hence $\delta\In{H}\Coo(V\times\Omega,\R)=:\R^{V\times\Omega}$.
Also note that for any non-zero $h\in H$ and for any $(v,x)\in V\times\Omega$,
by the fundamental theorem of calculus, we have 
\begin{equation}
\delta^{\vee}(h,v,x)=\frac{d^{\vee}(v+he_{i},x)-d^{\vee}(v,x)}{h}.\label{eq:deltaAsIncrRatio}
\end{equation}

We prove below that $\lim_{h\to0}\delta(h)=\frac{\partial d^{\vee}}{\partial e_{i}}|_{V\times\Omega}$
in the space $\R^{V\times(\Omega)}$ which has the underlying set
\[
\left|\R^{V\times(\Omega)}\right|:=\left\{ \phi\in\R^{V\times\Omega}\mid\phi^{\wedge}\In{V}\D(\Omega)\right\} ,
\]
 and figures defined by $p\In{W}\R^{V\times(\Omega)}$ iff $p^{\vee}:W\times V\times\Omega\ra\R$
is smooth and locally of uniformly bounded support with respect to
$W\times V$.

Since $d^{\vee}$ is locally of uniformly bounded support (Thm.~\ref{thm:2.3KR}),
we may assume that $V$ and $H$ are sufficiently small so that $\delta^{\vee}:H\times V\times\Omega\ra\R$
is of uniformly bounded support with respect to $H\times V$. Thus
\begin{equation}
\delta\In{H}\R^{V\times(\Omega)}.\label{eq:deltaInRVOm}
\end{equation}
 To prove the above mentioned limit equality, let $A$ be a $D$-open
subset of $\R^{V\times(\Omega)}$ such that $\frac{\partial d^{\vee}}{\partial e_{i}}|_{V\times\Omega}\in A$.
From \eqref{eq:deltaInRVOm} we know that $\delta^{-1}(A)=:B$ is
open in $H$. Moreover, $\delta(0)=\frac{\partial d^{\vee}}{\partial e_{i}}|_{V\times\Omega}\in A$
so $0\in B$. This proves that $\lim_{h\to0}\delta(h)=\frac{\partial d^{\vee}}{\partial e_{i}}|_{V\times\Omega}$
in $\R^{V\times(\Omega)}$.

Now, we apply this limit to the adjoint map 
\begin{equation}
(-)^{\wedge}:\phi\in\left|\R^{V\times(\Omega)}\right|\mapsto\phi^{\wedge}\in\left|\D(\Omega)^{V}\right|,\label{eq:adjointMap}
\end{equation}
 where the domain is the diffeological space $\R^{V\times(\Omega)}$,
and the codomain is the space $\D(\Omega)\uparrow V$ with $\left|\D(\Omega)^{V}\right|=|\Coo(V,\D(\Omega))|$
the underlying set and figures defined by $q\In{\tilde{W}}\D(\Omega)\uparrow V$
iff $\left(q^{\vee}\right)^{\vee}:\tilde{W}\times V\times\Omega\ra\R$
is smooth and locally of uniformly bounded support. We claim that
the adjoint map \eqref{eq:adjointMap} is also smooth with respect
to these diffeological structures on its domain and codomain. In fact,
if $p\In{W}\R^{V\times(\Omega)}$, then $\left(\left((-)^{\wedge}\circ p\right)^{\vee}\right)^{\vee}=p^{\vee}$
which is locally of uniformly bounded support by the definition of
the diffeology on $\R^{V\times(\Omega)}$. Therefore $(-)^{\wedge}:\R^{V\times(\Omega)}\ra\D(\Omega)\uparrow V$
is smooth and hence it is also $D$-continuous: 
\[
\left(\frac{\partial d^{\vee}}{\partial e_{i}}|_{V\times\Omega}\right)^{\wedge}=\left(\lim_{h\to0}\delta(h)\right)^{\wedge}=\lim_{h\to0}\delta(h)^{\wedge}\qquad\text{ in }\D(\Omega)\uparrow V.
\]

Now consider the evaluation at $v\in V\subseteq U$: 
\[
\text{ev}_{v}:\phi\in\left|\D(\Omega)\uparrow V\right|=\left|\D(\Omega)^{V}\right|\mapsto\phi(v)\in\left|\D(\Omega)\right|.
\]
We claim that $\text{ev}_{v}:\D(\Omega)\uparrow V\ra\D(\Omega)$ is
smooth. In fact, for any $q\In{\tilde{W}}\D(\Omega)\uparrow V$, i.e.,
\begin{equation}
\left(q^{\vee}\right)^{\vee}:\tilde{W}\times V\times\Omega\ra\R\text{ is smooth and locally of uniformly bounded support}.\label{eq:dAdjAdjIsLUBS}
\end{equation}
 We need to prove that $\left(\text{ev}_{v}\circ q\right)^{\vee}:\tilde{W}\times\Omega\ra\R$
is also smooth and locally of uniformly bounded support. Take $w\in\tilde{W}$,
and from \eqref{eq:dAdjAdjIsLUBS} we have open neighbourhoods $C$
of $w$ and $D$ of $v$ so that $\left(q^{\vee}\right)^{\vee}|_{C\times D\times\Omega}$
is of uniformly bounded support. We may assume that $\supp\left[\left(q^{\vee}\right)^{\vee}(w',v',-)\right]\subseteq K\Subset\Omega$
for all $(w',v')\in C\times D$. But $\left(q^{\vee}\right)^{\vee}(w',v',-)=q(w')(v')=\text{ev}_{v'}(q(w'))$.
Therefore, for all $w'\in C$ we have $\supp\left[\left(\text{ev}_{v}\circ q\right)^{\vee}(w',-)\right]=\supp\left[q(w')(v)\right]\subseteq K$.
By Cartesian closedness, $\text{ev}_{v}\circ q$ is smooth and hence
it is a figure of $\D(\Omega)$. This proves that $\text{ev}_{v}:\D(\Omega)\uparrow V\ra\D(\Omega)$
is smooth and hence it is also $D$-continuous. So we have: 
\begin{align*}
\frac{\partial d^{\vee}}{\partial e_{i}}(v,-) & =\text{ev}_{v}\left[\left(\frac{\partial d^{\vee}}{\partial e_{i}}|_{V\times\Omega}\right)^{\wedge}\right]=\text{ev}_{v}\left[\lim_{h\to0}\delta(h)^{\wedge}\right]\\
 & =\lim_{h\to0}\delta(h)^{\wedge}(v)=\lim_{h\to0}\frac{d^{\vee}(v+he_{i},-)-d^{\vee}(v,-)}{h}\\
 & =\lim_{h\to0}\frac{d(v+he_{i})-d(v)}{h}\qquad\forall v\in V.
\end{align*}
 Therefore, this limit exists in $\D(\Omega)$. By assumption, $l:\left|\D(\Omega)\right|\ra\R$
is $D$-continuous and linear, so 
\begin{align*}
l\left(\frac{\partial d^{\vee}}{\partial e_{i}}(v,-)\right) & =l\left[\lim_{h\to0}\frac{d(v+he_{i})-d(v)}{h}\right]\\
 & =\lim_{h\to0}\frac{l(d(v+he_{i}))-l(d(v))}{h}\\
 & =\frac{\partial\left(l\circ d\right)}{\partial e_{i}}(v).
\end{align*}
 This proves that the first partial derivatives of $l\circ d$ exist
and are continuous because both $l$ and $\frac{\partial d^{\vee}}{\partial e_{i}}$
are $D$-continuous. We can now apply the same procedure to the figure
\[
\left(\frac{\partial d^{\vee}}{\partial e_{i}}\right)^{\wedge}\In{U}\D(\Omega)
\]
 obtaining that also the second partial derivatives of $l\circ d$
exist and are continuous. By applying inductively this process, we
get the conclusion $l\circ d\in\Coo(U,\R)$. Finally, from Thm. \ref{thm:DIsCVS}
we have $l\in\D'(\Omega)$.
\end{proof}
We also have the following
\begin{cor}
\label{cor:equivDistr}Let $l:\left|\D(\Omega)\right|\ra\R$ be a
linear map. Then the following are equivalent: 
\begin{enumerate}[leftmargin={*},label=(\roman*),align=left]
\item \label{enu:equivDistr1}$l$ is continuous in the locally convex
topology, i.e., it is a distribution. 
\item \label{enu:equivDistr2}$l$ is continuous in the $D$-topology on
$\D(\Omega)$. 
\item \label{enu:equivDistr3}$l\in\Coo(\D(\Omega),\R)$
\end{enumerate}
\end{cor}
\begin{proof}
\ref{enu:equivDistr1} $\Rightarrow$ \ref{enu:equivDistr2}: From
Cor.~\ref{cor:D-TopFinerLCTop}; \ref{enu:equivDistr2} $\Rightarrow$
\ref{enu:equivDistr3}: From Thm.~\ref{thm:DContImpliesDistr}; \ref{enu:equivDistr3}
$\Rightarrow$ \ref{enu:equivDistr1}: From Thm.~\ref{thm:DIsCVS}.
\end{proof}
From the proof of Thm. \ref{thm:DContImpliesDistr} we have
\begin{cor}
Let $U$ be an open set in $\R^{n}$ and let $d\In{U}\D(\Omega)$.
Then $d$ is smooth in the usual sense, i.e., for all $\alpha\in\N^{n}$
the partial derivative $\partial^{\alpha}d:U\ra\left|\D(\Omega)\right|$
exists as the limit of a suitable incremental ratio in the topological
vector space $\left|\D(\Omega)\right|$ with the $D$-topology. Moreover,
$\partial^{\alpha}d\In{U}\D(\Omega)$.
\end{cor}
By applying this result to a curve $d\In{\R}\D(\Omega)$, and knowing
that the $D$-topology is finer than the usual locally convex topology,
we get an independent proof that $\D(\Omega)$ is a CVS.

We close this section with the following result, which underscores
the difference between $\D(\Omega)$ and its counterpart $\Dsmooth(\Omega)$;
in its statement, if $F\in\Diff$ is also a vector space, then we
set 
\[
F'_{\text{{\rm s}}}:=\left(\left\{ l\in\Coo(F,\R)\mid l\text{ is linear}\right\} \prec\Coo(F,\R)\right)
\]
 for its smooth dual space (this notation has been used for the special
cases in Thm.~\ref{thm:DtopAndTVS}).
\begin{cor}
\label{cor:DandDsmooth}\ 
\begin{enumerate}[leftmargin={*},label=(\roman*),align=left]
\item \label{enu:D'eqD'smooth}$\left|\D'(\Omega)\right|=\left|\D(\Omega)'_{\text{{\rm s}}}\right|$
and $\D'(\Omega)\supseteq\D(\Omega)'_{\text{{\rm s}}}$. 
\item \label{enu:D'notEqDsmooth'smooth}$\left|\D'(\Omega)\right|\supseteq\left|\Dsmooth(\Omega)'_{\text{{\rm s}}}\right|=\left\{ l\in\Coo(\Dsmooth(\Omega),\R)\mid l\text{ is linear}\right\} $. 
\end{enumerate}
\end{cor}
\begin{proof}
\ref{enu:D'eqD'smooth}: We first prove that the underlying sets are
equal, i.e., $\left|\D'(\Omega)\right|=\left|\D(\Omega)'_{\text{{\rm s}}}\right|$.
In fact, this follows from the equivalence \ref{enu:equivDistr1} $\iff$
\ref{enu:equivDistr3} of Cor.~\ref{cor:equivDistr}. Now, if $d\In{U}\D(\Omega)'_{\text{{\rm s}}}$,
then $d^{\vee}:U\times\D(\Omega)\ra\R$ is smooth. The space $\D'(\Omega)$
is functionally generated by all linear functionals $l:\left|\D'(\Omega)\right|\ra\R$
which are continuous with respect to the weak$^{*}$ topology. Since
each one of these functionals is of the form $l=\text{ev}_{\phi}$
for some $\phi\in\D(\Omega)$, we only need to consider $\left(\text{ev}_{\phi}\circ d\right)(u)=\text{ev}_{\phi}\left[d(u)\right]=d(u)(\phi)=d^{\vee}(u,\phi)$
for each $u\in U$. Therefore, $l\circ d=\text{ev}_{\phi}\circ d=d^{\vee}(-,\phi)$
is smooth, which implies that $d\In{U}\D'(\Omega)$.\\
 \ref{enu:D'notEqDsmooth'smooth}: As a consequence of Thm.~\ref{thm:2.3KR},
we know that $\D(\Omega)\subseteq\Dsmooth(\Omega)$. Therefore, if
$l\in\Coo(\Dsmooth(\Omega),\R)$ is linear, then we also have $l\in\Coo(\D(\Omega),\R)$.
Now $l\in\left|\D'(\Omega)\right|$ follows from Cor. \ref{cor:equivDistr}.
\end{proof}

\section{Spaces for Colombeau generalized functions as diffeological spaces}

It is natural to view all the spaces used to define CGF as diffeological
spaces. We will start with $\Coo(\Omega)^{I}$, $\ems(\Omega)$, $\mathcal{A}_{q}(\Omega)$,
$U(\Omega)$, $\mathcal{E}^{e}(\Omega)$ and $\emf(\Omega)$, with
the aim to prove that also the quotient spaces $\gs(\Omega)$, $\gf(\Omega)$
are smooth differential algebras.

\subsection*{The space $\Coo(\Omega)^{I}$}

Elements $(u_{\eps})$ of $\Coo(\Omega)^{I}$ are arbitrary nets,
indexed in $\eps\in I$, of smooth functions on $\Omega$. There are
studies of Colombeau-like algebras with smooth or continuous $\eps$-dependence
(see \cite{BuKu12,GiKu13} and references therein). In \cite{GiKu14}
it has been proved that a very large class of equations have no solution
if we request continuous dependence with respect to $\eps\in I$.
For this reason, it is natural to think $I$ as a space with the discrete
diffeology (see \ref{enu:discreteDiff} of Rem.~\ref{rem:propDiffSpaces}),
i.e., where only locally constant maps $d:U\ra I$ are figures
$d\In{U}I$. With this structure, the space $I$ is functionally generated
by $\Set(I,\R)$. If we think of $\Coo(\Omega)$ as the space $\Coo(\Omega,\R)\in\Diff$,
then by Cartesian closedness (Thm.~\ref{thm:FgenIsCartClosed}) we
have $u\in\left|\Coo(\Omega)^{I}\right|$, i.e., $u\in\Coo(I,\Coo(\Omega,\R))$,
iff $u\in\Set(I,\Coo(\Omega))$. The space $\Coo(\Omega)^{I}$ with
this diffeological structure will be denoted by $\Coo(\Omega,\R)^{I}$.
Figures $d\In{U}\Coo(\Omega,\R)^{I}$ are maps $d:U\ra\Set(I,\Coo(\Omega))$
such that $\left(d^{\vee}\right)^{\vee}(-,\eps,-)\in\Coo(U\times\Omega,\R)$
for all $\eps\in I$.

\subsection*{The space $\ems(\Omega)$}

The natural diffeology on 
\begin{multline*}
\ems(\Omega)=\{(u_{\eps})\in\Coo(\Omega)^{I}\mid\forall K\Subset\Omega\,\forall\alpha\in\N^{n}\,\exists N\in\N:\\
\sup_{x\in K}|\partial^{\alpha}u_{\eps}(x)|=O(\eps^{-N})\}
\end{multline*}
 is the sub-diffeology of $\Coo(\Omega,\R)^{I}$: 
\[
\Ems(\Omega):=\left(\ems(\Omega)\prec\Coo(\Omega,\R)^{I}\right).
\]
Its figures $d\In{U}\Ems(\Omega)$ are maps $d:U\ra\ems(\Omega)$
such that $\left(d^{\vee}\right)^{\vee}(-,\eps,-)\in\Coo(U\times\Omega,\R)$
for all $\eps\in I$.

\subsection*{The space $\mathcal{A}_{q}(\Omega)$}

The set $\mathcal{A}_{0}(\Omega)=\left\{ \phi\in\left|\D(\Omega)\right|\mid\int\phi=1\right\} $
has a natural diffeology, the sub-diffeology of $\D(\Omega)$. So
\[
\boldsymbol{\mathcal{A}}_{0}(\Omega):=\left(\mathcal{A}_{0}(\Omega)\prec\D(\Omega)\right)\in\Diff.
\]
Analogously, the set 
\[
\mathcal{A}_{q}(\Omega)=\left\{ \phi\in\mathcal{A}_{0}(\Omega)\mid\forall\alpha\in\N^{n}:\ 1\le|\alpha|\le q\Rightarrow\int x^{\alpha}\phi(x)\diff{x}=0\right\} 
\]
has a natural diffeology, the sub-diffeology of $\boldsymbol{\mathcal{A}}_{0}(\Omega)$
So 
\[
\boldsymbol{\mathcal{A}}_{q}(\Omega):=\left(\mathcal{A}_{q}(\Omega)\prec\boldsymbol{\mathcal{A}}_{0}(\Omega)\right)=\left(\mathcal{A}_{q}(\Omega)\prec\D(\Omega)\right)\in\Diff,
\]
where we used the property $\left(S\prec\left(T\prec X\right)\right)=\left(S\prec X\right)$
if $S\subseteq T\subseteq|X|$ and $X\in\Diff$. Therefore, figures
$d\In{U}\boldsymbol{\mathcal{A}}_{q}(\Omega)$ are maps $d:U\ra\mathcal{A}_{q}(\Omega)$
such that $d^{\vee}\in\Coo(U\times\Omega,\R)$ and $d^{\vee}$ is
locally of uniformly bounded support (Thm.~\ref{thm:2.3KR}).

Note that $\int_{\Omega}:\D(\Omega)\ra\R$ is a linear and diffeologically
smooth map. Therefore, $\mathcal{A}_{0}(\Omega)$ is an affine space
which is closed in the locally convex topology of $\D(\Omega)$. An
isomorphism with the corresponding vector space $\mathcal{A}_{00}(\Omega):=\ker\left(\int_{\Omega}\right)$
is given by $\phi\in\mathcal{A}_{0}(\Omega)\mapsto\phi-\phi_{0}\in\mathcal{A}_{00}(\Omega)$,
where $\phi_{0}\in\mathcal{A}_{0}(\Omega)$ is any fixed element.
This isomorphism is clearly diffeologically smooth. This solves the
problem stated in \ref{enu:A_0Affine} of Rem. \ref{rem:catFrameworksGF}.

\subsection*{The space $U(\Omega)$}

In Def.\ \ref{def:fullCA} of the full Colombeau algebra, the set
$U(\Omega)$ serves as domain of the representatives $R:U(\Omega)\ra\R$
of CGF in $\gf(\Omega)$. These representatives are requested to be
smooth in the $\Omega$ slot (note that $U(\Omega)\subseteq\mathcal{A}_{0}\times\Omega)$
but with no particular regularity in the $\mathcal{A}_{0}$ slot (which
serves as an index set for the full Colombeau algebra, analogous to
the interval $I$ as an index set for the special one). This means
that we shall consider the discrete diffeology on $\mathcal{A}_{0}$
and the standard diffeology on $\Omega$. If we identify the set $\mathcal{A}_{0}$
with the corresponding diffeological space with the discrete diffeology,
then 
\begin{align*}
\boldsymbol{U}(\Omega): & =\left(\left\{ (\phi,x)\in\mathcal{A}_{0}\times\Omega\mid\text{supp}(\phi)\subseteq\Omega-x\right\} \prec\mathcal{A}_{0}\times\Omega\right)\\
 & =\left(U(\Omega)\prec\left|\D(\R^{n})\right|\times\R^{n}\right)\in\Diff,
\end{align*}
where we used the property $(A\prec D)\times(O\prec R)=(A\times O\prec D\times R)$
which holds in $\Diff$. Therefore, figures $d\In{V}\boldsymbol{U}(\Omega)$
are maps $d:V\ra U(\Omega)$ such that the two projections verify
$d_{1}\in\Set(V,\left|\D(\R^{n})\right|)$ and $d_{2}\in\Coo(V,\Omega)$.

\subsection*{The space $\mathcal{E}^{e}(\Omega)$}

The space $\mathcal{E}^{e}(\Omega)$ (see Def. \ref{def:fullCA})
inherits its diffeological structure from $\Coo(\boldsymbol{U}(\Omega),\R)\in\Diff$:
\[
\boldsymbol{\mathcal{E}}^{e}(\Omega):=\left(\mathcal{E}^{e}(\Omega)\prec\Coo(\boldsymbol{U}(\Omega),\R)\right).
\]
Figures $d\In{V}\boldsymbol{\mathcal{E}}^{e}(\Omega)$ are maps $d:V\ra\mathcal{E}^{e}(\Omega)$
such that $d^{\vee}\in\Coo(V\times\boldsymbol{U}(\Omega),\R)$. We
give an equivalent characterization of $\boldsymbol{\mathcal{E}}^{e}(\Omega)$
as follows: For $\phi\in\mathcal{A}_{0}$, set 
\[
\Omega_{\phi}:=\Omega\cap\left\{ x\in\R^{n}\mid\text{supp}(\phi)\subseteq\Omega-x\right\} .
\]
As a convention, when $\Omega_{\phi}=\emptyset$, we think of $\Coo(\Omega_{\phi},\R)$
as a set with a single element. Since $R\in\mathcal{E}^{e}(\Omega)$
iff $R^{\wedge}:\mathcal{A}_{0}\ra\bigcup_{\phi\in\mathcal{A}_{0}}\Coo(\Omega_{\phi},\R)$
and $R(\phi,-)\in\Coo(\Omega_{\phi},\R)$ for all $\phi\in\mathcal{A}_{0}$,
$R^{\wedge}\in\prod_{\phi\in\mathcal{A}_{0}}\Coo(\Omega_{\phi},\R)$.
By Cartesian closedness of $\Diff$: 
\[
\boldsymbol{\mathcal{E}}^{e}(\Omega)\simeq\prod_{\phi\in\mathcal{A}_{0}}\Coo(\Omega_{\phi},\R).
\]
 Therefore, up to smooth isomorphism, figures of $\boldsymbol{\mathcal{E}}^{e}(\Omega)$
can be described as maps $d:V\ra\prod_{\phi\in\mathcal{A}_{0}}\Coo(\Omega_{\phi},\R)$
such that $d(-)(\phi)^{\vee}\in\Coo(V\times\Omega_{\phi},\R)$ for
all $\phi\in\mathcal{A}_{0}$.

\subsection*{The space $\emf(\Omega)$}

The natural diffeology on the space of moderate functions $\emf(\Omega)$
is the sub-diffeology of $\boldsymbol{\mathcal{E}}^{e}(\Omega)$.
Hence, 
\[
\Emf(\Omega):=\left(\emf(\Omega)\prec\boldsymbol{\mathcal{E}}^{e}(\Omega)\right)\in\Diff.
\]
Figures $d\In{V}\Emf(\Omega)$ are maps $d:V\ra\emf(\Omega)$ such
that $d(-)(\phi,-)\in\Coo(V\times\Omega_{\phi},\R)$ for all $\phi\in\mathcal{A}_{0}$.

\subsection*{The special and full Colombeau algebras}

Since the category $\Diff$ of diffeological spaces is cocomplete,
both quotient algebras $\gs(\Omega)$ and $\gf(\Omega)$ can be viewed
as objects of $\Diff$: 
\begin{align*}
\Gs(\Omega): & =\Ems(\Omega)/\ns(\Omega)\\
\Gf(\Omega): & =\Emf(\Omega)/\nf(\Omega).
\end{align*}
Figures of these spaces can be described using the notion of quotient
diffeology. E.g.\ $d\In{U}\Gs(\Omega)$ iff $d:U\ra\gs(\Omega)$
and for any $u\in U$ we can find an open neighbourhood $V$ of $u$
in $U$ and a map $\delta:V\ra\Coo(\Omega)^{I}$ such that
\begin{enumerate}[leftmargin={*},label=(\roman*),align=left]
\item $\delta(v)$ is moderate for all $v\in V$ 
\item $\left(\delta^{\vee}\right)^{\vee}(-,\eps,-)\in\Coo(V\times\Omega,\R)$
for all $\eps\in I$ 
\item $d|_{V}=\pi\circ\delta$, where $\pi:(u_{\eps})\in\ems(\Omega)\mapsto[u_{\eps}]\in\gs(\Omega)$
is the projection on the quotient. 
\end{enumerate}
Analogously, we can describe figures of the full Colombeau algebra.

\bigskip{}

We can now state the following natural result:
\begin{thm}
\label{thm:CGASmooth}Both for the special and the full Colombeau
algebras $\mathcal{G}(\Omega)\in\left\{ \gs(\Omega),\gf(\Omega)\right\} $,
the sum, product and derivation maps 
\begin{align*}
+:\mathcal{G}(\Omega)\times\mathcal{G}(\Omega) & \ra\mathcal{G}(\Omega)\\
\cdot:\mathcal{G}(\Omega)\times\mathcal{G}(\Omega) & \ra\mathcal{G}(\Omega)\\
\partial^{\alpha}:\mathcal{G}(\Omega) & \ra\mathcal{G}(\Omega)\qquad\forall\alpha\in\N^{n}
\end{align*}
are smooth. Therefore, with respect to the $D$-topology, $\mathcal{G}(\Omega)$
is a topological algebra with continuous derivations.

Moreover, if $\left(\psi_{\eps}\right)\in\D(\Omega)^{I}$ is a net
verifying properties \ref{enu:emb_supp}, \ref{enu:emb_int1}, \ref{enu:emb_moderate},
\ref{enu:emb_moments}, \ref{enu:emb_intAbs} of Thm.~\ref{thm:magicMollifier},
and let $\iota_{\Omega}$ be defined as in \eqref{eq:embeddingSpecial}
and let $\sigma_{\Omega}(f):=[f]\in\mathcal{G}(\Omega)$ for all $f\in\Coo(\Omega)$,
then the embeddings 
\begin{align*}
\iota_{\Omega}:\left|\D'(\Omega)\right| & \ra\mathcal{G}(\Omega)\\
\sigma_{\Omega}:\Coo(\Omega) & \ra\mathcal{G}(\Omega)
\end{align*}
are smooth maps if we equip $\left|\D'(\Omega)\right|$ with the sub-diffeology
of $\Coo(\D(\Omega),\R)$.\end{thm}
\begin{proof}
We prove that the maps are smooth for the case $\mathcal{G}(\Omega)=\gs(\Omega)$,
since the proof is similar for the case $\mathcal{G}(\Omega)=\gf(\Omega)$.

Concerning the smoothness of the sum map, let $d\In{U}\gs(\Omega)\times\gs(\Omega)$,
i.e., $p_{i}\circ d\In{U}\gs(\Omega)$, where $p_{i}:\gs(\Omega)\times\gs(\Omega)\ra\gs(\Omega)$,
$i=1,2$, are the projections. Hence, by the definition of the quotient
diffeology on $\gs(\Omega)$, for any $u\in U$ we can write $\left(p_{i}\circ d\right)|_{V_{i}}=\pi\circ\delta_{i}$,
where $V_{i}\in\tau_{U}$, $u\in V_{1}\cap V_{2}$, $\delta_{i}\In{V_{i}}\Ems(\Omega)$.
Thus, we can write the composition 
\[
\left(+\circ d\right)|_{V_{1}\cap V_{2}}:v\mapsto\pi\left[\delta_{1}(v)\right]+\pi\left[\delta_{2}(v)\right]=\pi\left[\delta_{1}(v)+\delta_{2}(v)\right].
\]
 Since $\left(\delta_{1}+\delta_{2}\right)|_{V_{1}\cap V_{2}}\In{V_{1}\cap V_{2}}\Ems(\Omega)$,
the conclusion follows from the definition of the quotient diffeology.

Analogously, we can prove that the product map is smooth.

Concerning the smoothness of the partial derivative $\partial^{\alpha}$,
if $d\In{U}\gs(\Omega)$, then for any $u\in U$ we can write $d|_{V}=\pi\circ\delta$,
where $u\in V\in\tau_{U}$ and $\delta\In{V}\Ems(\Omega)$. Therefore,
we have 
\[
\left(\partial^{\alpha}\circ d\right)|_{V}:v\mapsto\partial^{\alpha}\left(d(v)\right)=\partial^{\alpha}\left(\pi(\delta(v))\right)=\pi\left[\partial^{\alpha}\delta(v)\right].
\]
But it is not difficult to show that $\partial^{\alpha}\in\Coo(\Ems(\Omega),\Ems(\Omega))$.
Hence, $\partial^{\alpha}\delta\In{V}\Ems(\Omega)$, and the conclusion
follows.

Concerning the smoothness of the embeddings, we only need to prove
that $\iota_{\Omega}$ is smooth, since the smoothness of $\sigma_{\Omega}$
follows directly from the definition of figures of a quotient diffeology.
Let $d\In{U}\left(\left|\D'(\Omega)\right|\prec\Coo(\D(\Omega),\R)\right)$,
i.e., $d^{\vee}\in\Coo(U\times\D(\Omega),\R)$. We can compute that
\begin{align*}
\left(\iota_{\Omega}\circ d\right)(u) & =\left[d(u)*(\eps\odot\psi_{\eps}|_{\Omega})\right]=\\
 & =\left[x\in\Omega\mapsto\langle d(u),\frac{1}{\eps^{n}}\psi_{\eps}\left(\frac{x-?}{\eps}\right)\rangle\right].
\end{align*}
For any fixed $\eps\in I$, we show below that the map 
\[
\delta_{\eps}:(u,x)\in U\times\Omega\mapsto\langle d(u),\frac{1}{\eps^{n}}\psi_{\eps}\left(\frac{x-?}{\eps}\right)\rangle=d^{\vee}\left[u,\frac{1}{\eps^{n}}\psi_{\eps}\left(\frac{x-?}{\eps}\right)\right]\in\R
\]
is smooth (for the moderateness property, see \cite{GKOS,SteVic09}).
Define the maps 
\begin{align}
S:(\eps,\phi)\in I\times\D(\R^{n}) & \mapsto\eps\odot\phi\in\D(\R^{n})\label{eq:S-Ttilde}\\
\widetilde{T}:(x,\phi)\in\R^{n}\times\D(\R^{n}) & \mapsto\phi(x-?)\in\D(\R^{n}).\label{eq:Ttilde}
\end{align}
Then 
\[
\delta_{\eps}(u,x)=d^{\vee}\left[u,S\left(\eps,\widetilde{T}(x,\psi_{\eps})\right)\right]\qquad\forall(u,x)\in U\times\Omega.
\]
Therefore, $\delta_{\eps}$ is smooth once we prove that both maps
$S$ and $\widetilde{T}$ are diffeologically smooth. This is done
in the following lemma:
\begin{lem}
\label{lem:S-T-smooth}The maps defined in \eqref{eq:S-Ttilde}, \eqref{eq:Ttilde}
and 
\begin{align*}
T:(x,\phi)\in\R^{n}\times\D(\R^{n}) & \mapsto\phi(?-x)\in\D(\R^{n})
\end{align*}
 are diffeologically smooth, i.e., $S\in\Coo(I\times\D(\R^{n}),\D(\R^{n}))$,
$\widetilde{T}\in\Coo(\R^{n}\times\D(\R^{n}),\D(\R^{n}))$, $T\in\Coo(\R^{n}\times\D(\R^{n}),\D(\R^{n}))$.\end{lem}
\begin{proof}
We only proceed for $S$, since the other two cases are similar. If
$\eps\in I$ and $p\In{U}\D(\R^{n})$, then $p^{\vee}\in\Coo(U\times\R^{n},\R)$
and $p^{\vee}$ is locally of uniformly bounded support with respect
to $U$ (Thm.~\ref{thm:2.3KR}). But $\left[S(\eps,-)\circ p\right]^{\vee}(u,x)=\frac{1}{\eps^{n}}p^{\vee}\left(u,\frac{x}{\eps}\right)$
for all $(u,x)\in U\times\Omega$, and this shows that $\left[S(\eps,-)\circ p\right]^{\vee}\in\Coo(U\times\R^{n},\R)$
and it is locally of uniformly bounded support with respect to $U$.
\end{proof}
Since $I$ has the discrete diffeology, all these $\delta_{\eps}$'s
together induce a smooth map $\delta:U\ra\Ems(\Omega)$ such that
$\pi\circ\delta=\iota_{\Omega}\circ d$. By the definition of the
quotient diffeology on $\gs(\Omega)$, the embedding $\iota_{\Omega}:|\D'(\Omega)|\ra\gs(\Omega)$
is smooth.
\end{proof}

\subsection{Colombeau ring of generalized numbers and evaluation of generalized
functions}

In this subsection we consider only the case of the special Colombeau
algebra $\gs(\Omega)$ since it is mostly studied in the literature. The
case of the full algebra can be treated analogously.

One of the main features of Colombeau theory is the possibility to
define a point evaluation of every CGF. Hence, it is natural to ask
whether this evaluation map 
\begin{equation}
\text{ev}:(u,x)\in\gs(\Omega)\times\otilc\mapsto u(x)\in\Rtil\label{eq:evalCGF}
\end{equation}
is a smooth map or not (see Section \ref{sub:Special-and-fullCA}
for the definitions of $\otilc$ and $\Rtil$). The diffeology we
consider on $\otilc$ and $\Rtil$ are the natural ones:
\begin{defn}
\label{def:diffStructureRtil}All the following are diffeological
spaces: 
\begin{enumerate}[leftmargin={*},label=(\roman*),align=left]
\item $\boldsymbol{R}_{M}:=\left(\R_{M}\prec\Coo(I,\R)\right)$ 
\item $\widetilde{\boldsymbol{R}}:=\boldsymbol{R}_{M}/\sim$ 
\item $\boldsymbol{\Omega}_{M}:=\left(\Omega_{M}\prec\Coo(I,\Omega)\right)$ 
\item $\widetilde{\boldsymbol{\Omega}}:=\boldsymbol{\Omega}_{M}/\sim$ 
\item $\widetilde{\boldsymbol{\Omega}}_{c}:=(\otilc\prec\widetilde{\boldsymbol{\Omega}})$ 
\end{enumerate}
\end{defn}
Note, e.g., that $d\In{U}\boldsymbol{\Omega}_{M}$ iff $d:U\ra\Omega_{M}$
and $d^{\vee}(-,\eps)\in\Coo(U,\Omega)$ for all $\eps\in I$.
\begin{thm}
\label{thm:evalCGF}The evaluation map \eqref{eq:evalCGF} is smooth.\end{thm}
\begin{proof}
Let $a\In{U}\gs(\Omega)$ and let $b\In{U}\widetilde{\boldsymbol{\Omega}}_{c}$.
We need to prove that $\text{ev}\circ\langle a,b\rangle\In{U}\widetilde{\boldsymbol{R}}$.
For any fixed $u\in U$, by definition of the quotient diffeologies,
we can write $a|_{V}=\pi_{1}\circ\alpha$ and $i\circ b|_{V}=\pi_{2}\circ\beta$,
where $u\in V\in\tau_{U}$, $\alpha\In{V}\ems(\Omega)$, $\beta\In{V}\boldsymbol{\Omega}_{M}$,
$i:\otilc\hookrightarrow\widetilde{\Omega}$ is the inclusion, and
$\pi_{1}:\ems(\Omega)\ra\gs(\Omega)$, $\pi_{2}:\Omega_{M}\ra\widetilde{\Omega}$
are the projections. Hence, for any $v\in V$, we have $\left(\text{ev}\circ\langle a,b\rangle\right)(v)=\text{ev}\left(a(v),b(v)\right)=\text{ev}\left(\pi_{1}(\alpha(v)),\pi_{2}(\beta(v))\right)=\text{ev}\left(\left[\left(\alpha^{\vee}\right)^{\vee}(v,\eps,-)\right],\left[\beta^{\vee}(v,\eps)\right]\right)=\left[\left(\alpha^{\vee}\right)^{\vee}(v,\eps,\beta^{\vee}(v,\eps))\right]$.
Note that for any $\eps\in I$, $\left(\alpha^{\vee}\right)^{\vee}(-,\eps,-)\in\Coo(V\times\Omega,\R)$
and $\beta^{\vee}(-,\eps)\in\Coo(V,\Omega)$, so the restriction $\left(\text{ev}\circ\langle a,b\rangle\right)|_{V}$
can be written as an ordinary smooth function defined on $V$ composed
with the projection $\pi:\R_{M}\ra\Rtil$. Therefore, $\text{ev}\circ\langle a,b\rangle\in_{U}\widetilde{\boldsymbol{R}}$.
\end{proof}

\section{Conclusions and open problems}

We explore why the categories $\Fr$ of Frölicher spaces, $\Diff$
of diffeological spaces and $\Fgen$ of functionally generated (diffeological)
spaces work as good frameworks both for the classical spaces of functional
analysis and for the Colombeau algebras. On the one hand, there seem
to be few differences between $\Fgen$ and $\Fr$: we can say that
the former seems better than the latter because in $\Fgen$ we don't
have the problem of extending to the whole space locally defined functionals;
but in the latter, it is easier to work directly with globally defined
functionals when the $D$-topology of the space is unknown. On the
other hand, the usual counter-examples about locally Cartesian closedness
of $\Fr$ do not seem to work in $\Fgen$. Moreover, if compared to
$\Diff$, functionally generated spaces seem to be closer to spaces
used in functional analysis, where testing smoothness using functionals
is customary. On the other hand, Thm.~\ref{thm:CGASmooth} and Thm.~\ref{thm:evalCGF},
show that $\Diff$ can be considered a promising categorical framework
for Colombeau algebras. Some open problems underscored by the present
work are the following: 
\begin{itemize}
\item A clear and useful example of functionally generated space which is
not Frölicher and where locally defined functionals cannot be extended
to the whole space is missing.
\item The problem to show that $\Diff$ gives also a sufficiently simple
infinite dimensional calculus for the diffeomorphism invariant Colombeau
algebra (see \cite{GKOS}) remains open. In particular, we note that
the differentiable uniform boundedness principle (\cite[Thm.~2.2.7]{GKOS})
is used in \cite{GKOS} only to prove the analogy of Lem.~\ref{lem:S-T-smooth},
whereas the other results of \cite[Section~2.2.1]{GKOS} seem repeatable
in $\Diff$ without the need to know the calculus on convenient vector
spaces. 
\item The relationship between the locally convex topology and the $D$-topology
on $\D(\Omega)$ is only partially solved (see Cor. \ref{cor:D-TopFinerLCTop}).
\item The relationship between the space of Schwartz distributions $\D'(\Omega)$
and the smooth dual $\D(\Omega)'_{\text{s}}$ is only partially solved
(see Cor. \ref{cor:DandDsmooth}). 
\item The preservation of colimits from the category of smooth manifolds to $\Fgen$
is only partially solved.\end{itemize}

\end{document}